\documentclass[10pt,reqno]{amsart}

\usepackage{enumerate}
\usepackage{amsmath, amssymb, amsthm}
\usepackage{mathrsfs}
\usepackage{esint}
\usepackage{xcolor}
\usepackage{mathtools}
\usepackage{bm}
\usepackage{accents}
\usepackage{hyperref}
\usepackage[foot]{amsaddr}
\usepackage{todonotes}

\makeatletter

\numberwithin{equation}{section}

\newtheorem{thm}{Theorem}[section]
\newtheorem{theorem}[thm]{Theorem}
\newtheorem{lemma}[thm]{Lemma}
\newtheorem{corollary}[thm]{Corollary}
\newtheorem{lem}[thm]{Lemma}
\newtheorem{prop}[thm]{Proposition}
\theoremstyle{definition}

\newtheorem{assumption}[thm]{Assumption}
\theoremstyle{remark}

\newtheorem{remark}[thm]{\bf{Remark}}

\newcommand\aint{-\hspace{-0.38cm}\int}

\newcommand\bE{\mathbb{E}}

\newcommand\bH{\mathbb{H}}

\newcommand\bM{\mathbb{M}}
\newcommand\bN{\mathbb{N}}

\newcommand\bR{\mathbb{R}}

\newcommand\bT{\mathbb{T}}

\newcommand\bZ{\mathbb{Z}}

\newcommand\cA{\mathcal{A}}
\newcommand\cB{\mathcal{B}}
\newcommand\cC{\mathcal{C}}

\newcommand\cF{\mathcal{F}}

\newcommand\cL{\mathcal{L}}
\newcommand\cM{\mathcal{M}}

\newcommand\cS{\mathcal{S}}
\newcommand\cT{\mathcal{T}}

\begin{document}

\title{$L_p$-estimates for nonlocal equations with general L\'evy measures}

\author{Hongjie Dong$^{1}$}
\address{$^1$ Division of Applied Mathematics, Brown University, 182 George Street, Providence, RI 02912, USA}
\email{Hongjie\_Dong@brown.edu}
\thanks{H. Dong was partially supported by the NSF under agreement DMS-2350129.}

\author{Junhee Ryu$^{2}$}
\address{$^2$ School of Mathematics, Korea Institute for Advanced Study, 85 Hoegi-ro, Dongdaemun-gu, Seoul, 02455, Republic of Korea}
\email{junhryu@kias.re.kr}
\thanks{J. Ryu was supported by a KIAS Individual Grant (MG101501) at Korea Institute for Advanced Study.}



\begin{abstract}
We consider nonlocal operators of the form
\begin{equation*}
    L_t u(x) = \int_{\bR^d} \left( u(x+y)-u(x)-\nabla u(x)\cdot y^{(\sigma)} \right) \nu_t(dy),
\end{equation*}
where $\nu_t$ is a general L\'evy measure of order $\sigma \in(0,2)$. We allow this class of L\'evy measures to be very singular and impose no regularity assumptions in the time variable. Continuity of the operators and the unique strong solvability of the corresponding nonlocal parabolic equations in $L_p$ spaces are established. We also demonstrate that, depending on the ranges of $\sigma$ and $d$, the operator can or cannot be treated in weighted mixed-norm spaces.
\end{abstract}

\maketitle

\section{Introduction}

In this paper, we study the parabolic equation
\begin{equation} \label{maineq_intro}
\begin{cases}
\partial_t u(t,x)=L_t u(t,x)+f(t,x),\quad &(t,x)\in \bR^d_T,
\\
u(0,x)=0,\quad & x\in \bR^d,
\end{cases}
\end{equation}
where $\bR^d_T:=(0,T)\times \bR^d$ and $L_t $ is a time-inhomogeneous $\sigma$-stable-like nonlocal operator with $\sigma\in(0,2)$. More specifically, $L_t$ is defined by
\begin{align} \label{op_para}
L_t u(t,x) = \int_{\bR^d} \left( u(t,x+y)-u(t,x)-\nabla u(t,x)\cdot y^{(\sigma)} \right) \nu_t(dy),
\end{align}
where $\sigma\in(0,2)$, $\nu_t$ is a $\sigma$-stable-like L\'evy measure (see Assumption \ref{levy}), and
\begin{equation*}
  y^{(\sigma)}:=\left(1_{\sigma\in(1,2)} + 1_{\sigma=1}1_{|y|\leq1} \right) y.
\end{equation*}

Nonlocal operators are closely related to probability theory. It is well known that the fractional Laplacian $-(-\Delta)^{\sigma/2}$ is the infinitesimal generator of a standard isotropic $\sigma$-stable L\'evy process $X_t$, which can be expressed by
\begin{equation*}
    -(-\Delta)^{\sigma/2}u(x) = \lim_{t\to0^+} \frac{\bE u(x+X_t) - u(x)}{t}.
\end{equation*}
In general, the infinitesimal generators of purely jump L\'evy processes are known to be of the form \eqref{op_para}. Conversely, by the L\'evy–Khintchine formula, given an arbitrary L\'evy measure, one can construct a L\'evy process whose infinitesimal generator coincides with such an operator (see e.g. \cite{Sato13}). Motivated by this correspondence, we introduce a large class of operators of order $\sigma\in(0,2)$, which includes singular operators, and study Sobolev type estimates for the associated equations.

Here are some explicit examples of nonlocal operators satisfying Assumption \ref{levy} below.
\begin{enumerate}
    \item $\nu_t(dy)=|y|^{-d-\sigma}dy$,

    \item $\nu_t(dy)=\sum_{i=1}^d |y_i|^{-1-\sigma}dy_i\otimes\delta_0^{d-1}(dy^i)$, where $\delta_0^{d-1}$ is the $(d-1)$-dimensional Dirac measure and $y^i$ denotes the vector obtained by removing the $i$-th component $y_i$ of $y$,

    \item The last example is even more singular than $(2)$;
    \begin{align*}
        \nu_t(dy) = \sum_{i=1}^d \sum_{k\in\bZ} 2^{-k\sigma} &\Big(\left(\delta_{2^k}(dy_i) + \delta_{2^k}(-dy_i) \right)\otimes\delta_0^{d-1}(dy^i) \Big),
    \end{align*}
    where $\delta_{2^k}$ is the one-dimensional Dirac measure concentrated at $x=2^k$.
\end{enumerate}
The first example generates the fractional Laplacian up to a constant. The second example is a well-known singular nonlocal operator: the generator of $d$ independent one-dimensional symmetric stable L\'evy processes. In this case, the Fourier transform of the corresponding operator is given by
    \begin{equation*}
        \cF[L_tu](\xi) = -c\sum_{i=1}^d |\xi_i|^\sigma \cF[u](\xi).
    \end{equation*}
    Here, the symbol $\sum_{i=1}^d |\xi_i|^\sigma$ is singular in the sense that one cannot apply the Fourier multiplier theorem. For the last example, compared to $(2)$, even after fixing $d-1$ coordinates, the resulting one-dimensional symbol fails to satisfy the Fourier multiplier theorem. To the best of our knowledge, our main results provide the first regularity result for equations with very singular operators like $(3)$. Moreover, in the present paper, no regularity assumption is assumed in the time direction.

There has been extensive work on Sobolev estimates for parabolic equations with nonlocal operators. 
We first refer the reader to \cite{M92,MP14,DL23}, where the operators associated with measures of the form $\nu_t(dy)=K(t,y)|y|^{-d-\sigma} dy$ were introduced. In \cite{M92}, an $L_p$ estimate for \eqref{maineq_intro} was obtained using the Fourier multiplier theorem. In this work, the kernel $K(t,y)$ satisfies an ellipticity condition, and is homogeneous of order zero and smooth in $y$.
In \cite{MP14}, the same authors extended the earlier result to the case $K(t,y)\geq K_0(t,y)$, where $K_0$ satisfies the conditions in \cite{M92}. Their approach was to use a probabilistic method to represent the solution for \eqref{maineq_intro}.
 In \cite{DL23}, the authors derived a mean oscillation estimate of solution to obtain weighted mixed-norm estimates for the equations, where a purely analytic approach was introduced.
 We also refer the reader to \cite{AM25}, where the authors extended the results in \cite{DL23} to nonlocal elliptic systems. For results on the corresponding nonlocal elliptic equations, see \cite{DK12,KK15}. Further results on spatially dependent measures can be found in \cite{DJK23,DL23, X13_var}.

There has also been considerable interest in L\'evy measures beyond the prototype of $|y|^{-d-\sigma}dy$.
 In \cite{X13}, the authors obtained an $L_p$ estimate when $\nu_t=\nu$ is a time-independent L\'evy measure such that $\nu \geq \nu^{(\sigma)}$ for some nondegenerate $\sigma$-stable L\'evy measure. Here, $\nu^{(\sigma)}$ can be represented in the polar coordinates as
 \begin{equation*}
     d\nu^{(\sigma)}(dy) = r^{-1-\sigma} dr \mu(d\theta),
 \end{equation*}
 where $\mu$ is a nonnegative finite measure on the unit sphere $S^{d-1}$ in $\bR^d$. In \cite{CKP24}, the authors treated anisotropic fractional Laplacians including Example $(2)$. In both \cite{CKP24,X13}, their main approach is to use a probabilistic representation of the solution.
 Further developments in Sobolev regularity theory include the study of variable-order operators beyond the constant order $\sigma$. See \cite{CK23,JSS25,KKK13,KKK19,L23,MP19}. We also refer the reader to \cite{Bass09,CLU14,CS09,DK20,Kass09,KW24,SS16} for H\"older regularity results.

Now we introduce the main contributions of this paper. First, we establish solvability of \eqref{maineq_intro} in unweighted $L_p$ spaces for a substantially larger class of operators than those previously studied. Next, in the special case when $d=1$ and $\sigma\in(1,2)$, we derive the weighted-mixed norm estimates for the equations. In contrast, when $d\geq2$ or $d=1$ with $\sigma\in(0,1)$, we also show that general nonlocal operators cannot be treated in weighted spaces. The borderline case when $d=\sigma=1$ remains open.
Lastly, we prove the continuity of $L_t$, which plays a crucial role in the proof of existence via the method of continuity. In the local case, the operator $a_{ij}D_{ij}$ is bounded from $W_p^2(\bR^d)$ to $L_p(\bR^d)$ for arbitrary bounded measurable coefficients $a_{ij}$. However, the corresponding boundedness of nonlocal operators from the Bessel potential space $H_p^\sigma(\bR^d)$ to $L_p(\bR^d)$ is not straightforward, especially for singular measures.

  For the proof of the a priori estimates, we do not rely on probabilistic representations of solutions, unlike the aforementioned results for singular measures \cite{CKP24,X13}. As in \cite{DL23}, for the unweighted estimates, we employ a level set argument together with Lemma \ref{lem10161553}, commonly referred to as ``crawling of ink spots lemma''. For the weighted mixed-norm estimates in the special case $d=1$ and $\sigma\in(1,2)$, an iteration argument is used to obtain a mean oscillation estimate of $(-\Delta)^{\sigma/2}u$, where $u$ is a solution to \eqref{maineq_intro}.
In these procedures, as in most nonlocal equations, a careful analysis of tail terms is required. However, in our case,  the singularity of L\'evy measures prevents a direct application of the approach in \cite{DL23}. To address this difficulty, we use the result from \cite{DR86}, originally formulated to deal with a wide class of maximal and singular integral operators arising in harmonic analysis. In order to make this result applicable to our framework, we introduce a new notion of maximal function associated with general measures and verify that it satisfies the assumptions in \cite{DR86}.

  The paper is organized as follows. In Section \ref{Sec2}, we introduce the function spaces, precise assumptions of the operators, and the main results. In Section \ref{sec3}, we provide the proof of the unweighted result (Theorem \ref{mainthm}). The weighted mixed-norm result (Theorem \ref{thm_weight}) is handled in Section \ref{sec4}. In Appendix \ref{secA}, we prove miscellaneous lemmas used in the main proofs, including careful analysis for general measures.

We finish the introduction with notations used in this paper.
 We use $``:="$ or $``=:"$ to denote a definition. For any $a\in\bR$, $a_+:=\max\{a,0\}$. By $\bN$ and $\bZ$, we denote the natural number system and the integer number system, respectively. We denote $\bN_0:=\bN\cup\{0\}$. As usual, $\bR^d$ stands for the Euclidean space of points $x=(x_1,\dots,x_d)$, and we denote
$$
B_r(x):=\{y\in\bR^d : |x-y|<r\}.
$$
We write $\bR_T^{d}:=(0,T)\times \bR^d$ and $\bR_T:=\bR^1_T$.
We use $D^n_x u$ to denote the partial derivatives of order $n\in\bN_0$ with respect to the space variables.
By $\mathcal{F}$ and $\mathcal{F}^{-1}$, we denote the $d$-dimensional Fourier transform and the inverse Fourier transform respectively, i.e.,
\begin{equation*}
    \mathcal{F}[f](\xi) := \widehat{f}(\xi) := \int_{\bR^d} e^{-i \xi \cdot x} f(x) dx, \qquad \mathcal{F}^{-1}[f](x) := \frac{1}{(2\pi)^d} \int_{\bR^d} e^{i \xi \cdot x} f(\xi) d\xi.
\end{equation*}
For a domain $D\subset \bR^{d+1}$, we denote by $L_{p}(D)$ the set of all measurable functions $u$ defined on $D$ such that
\begin{equation*}
  \|u\|_{L_{p}(D)}:=\left( \int_{D} |u(t,x)|^p dxdt \right)^{1/p}<\infty, \quad p\in(1,\infty),
\end{equation*}
and
\begin{equation*}
  \|u\|_{L_{\infty}(D)}:=\sup_{(t,x)\in D}|u(t,x)|.
\end{equation*}
As usual, for $\tau>0$,
\begin{equation*}
    [u]_{C^{\tau/\sigma,\tau}(D)} := \sup_{\substack{(t,x),(s,y)\in D \\ (t,x)\neq (s,y)}}
\frac{|u(t,x) - u(s,y)|}{|t - s|^{\tau/\sigma} + |x - y|^{\tau}}
\end{equation*}
and
\begin{equation*}
    \|u\|_{C^{\tau/\sigma,\tau}(D)} := \|u\|_{L_{\infty}(D)} + [u]_{C^{\tau/\sigma,\tau}(D)}.
\end{equation*}

\section{Main results} \label{Sec2}

We first introduce the assumptions for the operators.
Let $\nu_t$ be a family of L\'evy measures on $\bR^d$, that is, for each $t\in \bR$, $\nu_t$ is a $\sigma$-finite (positive) measure on $\bR^d$ such that $\nu_t(\{0\})=0$ and
$$
\int_{\bR^d} \min\{1,|y|^2\} \, \nu_t(dy)<\infty.
$$

For the solvability of the equation, we impose the following nondegenerate $\sigma$-stable-like assumption on $L_t$.

\begin{assumption} \label{levy}
$(i)$ For any measurable set $A\subset\bR^d$, the mapping
\begin{equation*}
    t\to \nu_t(A)
\end{equation*}
is measurable.

$(ii)$ There exists $\Lambda \geq 1$ such that for any $r>0$,
\begin{equation} \label{equiv}
    \nu_t(B_r^c) \leq \Lambda r^{-\sigma}
\end{equation}

  $(iii)$ If $\sigma=1$, then
  \begin{align*}
    \int_{r_1\leq|y|\leq r_2} y \, \nu_t(dy)=0, \quad 0<r_1<r_2.
  \end{align*}

  $(iv)$ There exists $N_0>0$ such that for any $\xi\in\bR^d$,
  \begin{equation} \label{nonde}
    \int_{|\xi\cdot y|\leq 1} |\xi\cdot y|^2 \nu_t(dy) \geq N_0 |\xi|^\sigma.
  \end{equation}
\end{assumption}

For continuity of the operator, we only need the following assumption, which corresponds to Assumption \ref{levy} without \eqref{nonde}.

\begin{assumption} \label{upper}
$(i)$ For any measurable set $A\subset\bR^d$, the mapping
\begin{equation*}
    t\to \nu_t(A)
\end{equation*}
is measurable.

$(ii)$ There exists $\Lambda \geq 1$ such that for any $r>0$,
\begin{equation} \label{upper_mu}
    \nu_t(B_r^c) \leq \Lambda r^{-\sigma}
\end{equation}

  $(iii)$ If $\sigma=1$, then
  \begin{align} \label{cancel_mu}
    \int_{r_1\leq|y|\leq r_2} y \, \nu_t(dy)=0, \quad 0<r_1<r_2.
  \end{align}
\end{assumption}

Now we introduce function spaces. Recall that for $p\in(1,\infty)$ and $\sigma\in(0,\infty)$, the Bessel potential space is given by
\begin{equation*}
    H_p^\sigma(\bR^d):= \{u\in L_p(\bR^d): (1-\Delta)^{\sigma/2}u\in L_p(\bR^d)\}
\end{equation*}
and
\begin{equation*}
    \|u\|_{H_p^\sigma(\bR^d)}:= \|(1-\Delta)^{\sigma/2}u\|_{L_p(\bR^d)}.
\end{equation*}
Here,
\begin{equation*}
    (1-\Delta)^{\sigma/2}u= \cF^{-1}\left[(1+|\xi|^2)^{\sigma/2}\cF(u)(\xi)\right](x).
\end{equation*}
Next, we denote by $\bH_{p}^\sigma(S,T)$ the collection of functions such that
\begin{equation*}
    \|u\|_{\bH_p^\sigma(S,T)}:= \|u\|_{L_p((S,T);H_p^\sigma(\bR^d))} + \|\partial_t u\|_{L_p((S,T)\times\bR^d)} <\infty.
\end{equation*}
We write $u\in \bH_{p,0}^\sigma(S,T)$ if there exists a sequence of functions $u_n\in C_c^\infty([S,T]\times \bR^d)$ with $u_n(S,x)=0$ such that
\begin{equation*}
    \|u-u_n\|_{\bH_{p}^\sigma(S,T)} \to0 \text{ as } n\to\infty.
\end{equation*}
In the case when $S=0$, we omit $S=0$, i.e., $\bH_{p,0}^\sigma(T):= \bH_{p,0}^\sigma(0,T)$.

We also introduce spaces with Muckenhoupt weights. We say that a locally integrable nonnegative function $\omega$ on $\bR^d$ is in the $A_p(\bR^d)$ Muckenhoupt class of weights if
\begin{equation*}
    [\omega]_{A_p(\bR^d)}:= \sup_{r>0,x_0\in\bR^d} \left( \aint_{B_r(x_0)} \omega(x) dx \right) \left( \aint_{B_r(x_0)} (\omega(x))^{1/(1-p)} dx \right)^{p-1} <\infty.
\end{equation*}
Furthermore, for a constant $K_0>0$, we write $[\omega]_{q,p} \leq K_0$ if $\omega(t,x)=\omega_1(t)\omega_2(x)$ for some $\omega_1\in A_q(\bR)$ and $\omega_2\in A_p(\bR^d)$ satisfying
\begin{equation*}
    [\omega_1]_{A_q(\bR)}, [\omega_2]_{A_p(\bR^d)} \leq K_0.
\end{equation*}
For $\omega_1\in A_q(\bR)$ and $\omega_2\in A_p(\bR^d)$, we define $L_{p,\omega_2}(\bR^d):=L_p(\bR^d,\omega_2 dx)$ and $L_{q,p,\omega}(\bR^d_T):= L_q((0,T),\omega_1dt;L_{p,\omega_2}(\bR^d))$. We also denote $H_{p,\omega_2}^\sigma(\bR^d)$ by the class of tempered distributions satisfying
\begin{equation*}
    H_{p,\omega_2}^\sigma(\bR^d):= \{u\in L_{p,\omega_2}(\bR^d): (1-\Delta)^{\sigma/2}u\in L_{p,\omega_2}(\bR^d)\}.
\end{equation*}
Similar to $\bH_{p}^\sigma(T)$, we denote $\bH_{q,p,\omega}^\sigma(T)$ by the collection of functions such that
\begin{equation*}
    \|u\|_{\bH_{q,p,\omega}^\sigma(T)}:= \|u\|_{L_q((0,T),\omega_1 dt;H_{p,\omega_2}^\sigma(\bR^d))} + \|\partial_t u\|_{L_{q,p,\omega}(\bR^d_T)} <\infty.
\end{equation*}
We write $u\in \bH_{q,p,\omega,0}^\sigma(T)$ if there exists a sequence of functions $u_n\in C_c^\infty([0,T]\times \bR^d)$ with $u_n(0,x)=0$ such that
\begin{equation*}
    \|u-u_n\|_{\bH_{q,p,\omega}^\sigma(T)} \to 0 \text{ as } n\to\infty.
\end{equation*}

We now state our first main theorem of this paper, where the boundedness of operators and the solvability of equations in unweighted spaces are presented.

\begin{theorem} \label{mainthm}
  Let $d\geq1$, $\sigma\in(0,2)$, $T\in(0,\infty)$, $\lambda\geq0$, and $p\in(1,\infty)$. Suppose that $L_t$ and $\cL_t$ satisfy Assumptions \ref{levy} and \ref{upper}, respectively.

  $(i)$ For each $t>0$, the operators $L_t$ and $\cL_t$ are continuous from $H_{p}^{\sigma}(\bR^d)$ to $L_p(\bR^d)$. Moreover, for each $t>0$,
      \begin{equation} \label{oper_bdd}
        \|\cL_t v\|_{L_{p}(\bR^d)} \leq N(d,p,\sigma,\Lambda, N_0) \|L_tv\|_{L_{p}(\bR^d)},
    \end{equation}
    where $v\in H_{p}^{\sigma}(\bR^d)$. In particular, when $\sigma=1$, one can also consider $\nabla u$ instead of $L_tu$ or $\cL_t u$.
  
  $(ii)$ For any $u\in \bH_{p,0}^{\sigma}(T)$ satisfying
  \begin{equation} \label{maineq}
    \partial_t u=L_tu -\lambda u + f \text{ in } \bR^d_T,
  \end{equation}
  we have
  \begin{equation} \label{mainest}
    \|\partial_t u\|_{L_{p}(\bR_T^d)} + \|\cL_t u\|_{L_{p}(\bR_T^d)} + \lambda\|u\|_{L_{p}(\bR_T^d)} \leq N\|f\|_{L_{p}(\bR_T^d)},
  \end{equation}
  where $N$ depends only on $d,p,\sigma,\Lambda$, and $N_0$. In particular, when $\sigma=1$, one can also consider $\nabla u$ instead of $\cL_t u$.

  $(iii)$ For any $f\in L_{p}(\bR^d_T)$, there exists a unique solution $u\in \bH_{p,0}^\sigma(T)$ to \eqref{maineq}.
\end{theorem}

Next, we show that although our operator is allowed to be highly singular, equations can be handled in weighted spaces in the special case when $d=1$ and $\sigma\in(1,2)$.

\begin{theorem} \label{thm_weight}
      Let $d=1$, $\sigma\in(1,2)$, $T\in(0,\infty)$, $\lambda\geq0$, $p,q\in(1,\infty)$, $K_0>0$, and $[\omega]_{q,p}\leq K_0$. Suppose that $L_t$ and $\cL_t$ satisfy Assumptions \ref{levy} and \ref{upper}, respectively.

  $(i)$ 
  For each $t>0$, the operators $L_t$ and $\cL_t$ are continuous from $H_{p,\omega_2}^{\sigma}(\bR)$ to $L_{p,\omega_2}(\bR)$. Moreover, for each $t>0$,
      \begin{equation*}
        \|\cL_t v\|_{L_{p,\omega_2}(\bR)} \leq N(d,p,\sigma,\Lambda, K_0, N_0) \|L_tv\|_{L_{p,\omega_2}(\bR)},
    \end{equation*}
    where $v\in H_{p,\omega_2}^{\sigma}(\bR)$.  In particular, when $\sigma=1$, one can also consider $\nabla u$ instead of $L_tu$ or $\cL_t u$.
  
  $(ii)$ For any $u\in \bH_{q,p,\omega,0}^{\sigma}(T)$ satisfying
  \begin{equation} \label{weighteq}
    \partial_t u=L_tu -\lambda u + f \text{ in } \bR_T
  \end{equation}
  and any operator $\cL_t$ satisfying Assumption \ref{levy}, 
  we have
  \begin{equation} \label{weightest}
    \|\partial_t u\|_{L_{q,p,\omega}(\bR_T)} + \|\cL_t u\|_{L_{q,p,\omega}(\bR_T)} + \lambda\|u\|_{L_{q,p,\omega}(\bR_T)} \leq N\|f\|_{L_{q,p,\omega}(\bR_T)},
  \end{equation}
  where $N$ depends only on $p,\sigma,\Lambda, K_0$, and $N_0$.  In particular, when $\sigma=1$, one can also consider $\nabla u$ instead of $\cL_t u$.

  $(iii)$ For any $f\in L_{q,p,\omega}(\bR_T)$, there exists a unique solution $u\in \bH_{q,p,\omega,0}^\sigma(T)$ to \eqref{weighteq}.
\end{theorem}

\begin{remark}
    In Theorem \ref{thm_weight} $(i)$, if we further assume that $\cL_t$ satisfies Assumption \ref{levy}, then $L_t$ and $\cL_t$ are equivalent in the sense that for each $t>0$,
      \begin{equation*}
        N^{-1}\|L_tv\|_{L_{p,\omega_2}(\bR^d)} \leq \|\cL_t v\|_{L_{p,\omega_2}(\bR^d)} \leq N\|L_tv\|_{L_{p,\omega_2}(\bR^d)}.
    \end{equation*}
\end{remark}

\begin{remark}
Let $u\in \bH_{q,p,\omega,0}^\sigma(T)$ be a solution to \eqref{maineq}. Note that by extending $u$ to be zero for $t<0$, $u,\partial_t u \in L_q((-\infty,T),\omega_1dt; L_{p,\omega_2}(\bR^d))$. Since by the fundamental theorem of calculus,
\begin{equation*}
    u(t,x)=\int_0^t \partial_s u(s,x)ds,
\end{equation*}
we can apply the Minkowski inequality to get
\begin{equation*}
    \|u(t,\cdot)\|_{L_{p,\omega_2}(\bR^d)} \leq 
    \int_0^t \left\|\partial_s u(t,\cdot)\right\|_{L_{p,\omega_2}(\bR^d)} ds \leq T \bM_t(\left\|\partial_t u(t,\cdot)\right\|_{L_{p,\omega_2}(\bR^d)})(t),
\end{equation*}
where $\bM_t$ denotes the standard maximal function in $t$ (see \eqref{eq11192353}). Hence, by the Hardy-Littlewood theorem and \eqref{weightest},
\begin{equation} \label{eq110722228}
    \| u\|_{L_{q,p,\omega}(\bR_T^d)} \leq N\min\{T,\lambda^{-1}\} \|\partial_t u\|_{L_{q,p,\omega}(\bR_T^d)},
\end{equation}
where $N$ is independent of $T$ and $\lambda$.
\end{remark}

\begin{remark}
    In \cite{DL23}, the authors studied the equation in $\bH_{q,p,\omega,0}^\sigma(T)$, where the L\'evy measure $\nu(dy)$ is of prototype $|y|^{-d-\sigma}dy$. In this remark, we show that such weighted spaces cannot be considered for general L\'evy measures when $d\geq2$ or $d=1$ with $\sigma\in(0,1)$. Whether the case $d=1$ with $\sigma=1$ can be handled in $\bH_{q,p,\omega,0}^\sigma(T)$ remains open.

 Our aim is to show that there exist a (time-independent) operator $L$ and a weight $\omega=\omega(x)\in A_p(\bR^d)$ such that $L$ is not a bounded operator in $L_{p,\omega}(\bR^d)$ for some $p\in (1,\infty)$. Here, we note that once the claim is proved, the extrapolation theorem (see e.g. \cite[Theorem 2.5]{DK18}) implies that the same conclusion holds for every $p\in (1,\infty)$.

Let us consider
 \begin{equation*}
     \nu(dy) := \sum_{k\in\bZ} 2^{-k\sigma} \left(\delta_{2^k}(dy_1) + \delta_{2^k}(-dy_1)\right)\otimes\delta_0^{d-1}(dy') + |y'|^{1-d-\sigma}dy'\otimes\delta_0(dy_1),
 \end{equation*}
 where the corresponding operator is of the form
 \begin{equation*}
     Lv(x):= L_{x_1}v(x) -c(-\Delta_{x'})^{\sigma/2}v(x),
 \end{equation*}
 where $x=(x_1,x')$ and $c>0$. Our aim is to prove that if $l\in [\sigma p,d(p-1))$, then $L$ is not a bounded operator in $L_{p}(\bR^d,|x|^ldx)$. The range is nonempty for large $p>1$ when $d=2$ or $d=1$ with $\sigma\in(0,1)$. We also remark that under the range of $l$, $|x|^l dx$ is in the Muckenhoupt $A_p$ class (see e.g. \cite[Example 7.1.7]{G14}).

    Let $v\in C_c^\infty(\bR^d)$ be a nonnegative function such that $v=1$ on $[-1,1]^d$, and $\text{supp}(v)\subset [-2,2]^d$. Then for $x\in \bR^d$ such that $x'\in[-1,1]^{d-1}$ and $x_1>3$, one can compute that
    \begin{equation*}
        Lv(x) = L_{x_1}v(x) = \sum_{j\in\bZ} 2^{-j\sigma} v(x+2^j) \geq \sum_{k=2}^\infty 2^{-k\sigma} 1_{(2^k-1,2^k+1)}(x_1).
    \end{equation*}
   Thus,
    \begin{align*}
        \|Lv\|_{L_{p}(\bR^d,|x|^ldx)}^p &\geq N  \int_{(3,\infty)\times [-1,1]^{d-1}} \sum_{k=2}^\infty 2^{-k\sigma p} 1_{(2^k-1,2^k+1)}(x_1) |x|^l dx.
    \end{align*}
    Since
    \begin{equation*}
        \int_{(3,\infty)} \sum_{k=2}^\infty 2^{-k\sigma p} 1_{(2^k-1,2^k+1)}(x_1) |x|^l dx_1 \geq N\sum_{k=2}^\infty 2^{k(l-\sigma p)}=\infty
    \end{equation*}
    when $l-\sigma p\geq 0$, we deduce that $\|Lv\|_{L_{p}(\bR^d,|x|^ldx)}=\infty$.

    Now we show that $\|(-\Delta)^{\sigma/2} v\|_{L_{p}(\bR^d,|x|^ldx)} <\infty$. For $|x|>3\sqrt{d}$ and $x+y\in [-2,2]^d$,
    \begin{equation*}
        |y|\geq|x|-|x+y|\geq |x|-2\sqrt{d}\geq |x|/3.
    \end{equation*}
    Thus,
    \begin{align*}
        |(-\Delta)^{\sigma/2}v(x)| \leq \|v\|_{L_{\infty}(\bR^d)} \int_{x+y\in [-2,2]^d} |y|^{-d-\sigma} dy \leq N|x|^{-d-\sigma},
    \end{align*}
    which implies that
    \begin{equation} \label{eq11062133}
        \int_{|x|>3\sqrt{d}} |(-\Delta)^{\sigma/2}v(x)|^p |x|^l dx \leq N\int_{|x|>3\sqrt{d}} |x|^{l-dp -\sigma p} dx <\infty.
    \end{equation}
Since $v\in C_c^\infty(\bR^d)$, $(-\Delta)^{\sigma/2} v$ is a continuous function on $\bR^d$. This fact together with \eqref{eq11062133} yields $\|(-\Delta)^{\sigma/2} v\|_{L_{p}(\bR^d,|x|^ldx)} <\infty$.

\end{remark}

\section{Solvability of equations with general operators} \label{sec3}

\subsection{The case of $p=2$}

In this section, we prove Theorem \ref{mainthm} when $p=2$.

\begin{prop} \label{prop2}
  Theorem \ref{mainthm} holds when $p=2$.
\end{prop}

\begin{proof}
  $(i)$
  First, we prove that $L_t$ is continuous. Note that
  \begin{equation} \label{eq12041535}
    \widehat{L_tv}(t,\xi)=m(t,\xi)\widehat{v}(\xi),
  \end{equation}
  where
      \begin{equation*}
    m(t,\xi):=\int_{\bR^d} \left(e^{i\xi\cdot y}-1 - i\xi\cdot y^{(\sigma)}\right) \nu_t(dy).
  \end{equation*}
  Thus, by the Plancherel theorem and \eqref{ineq9251512},
  \begin{align*}
    \|L_tv\|_{L_2(\bR^d)}=\|\widehat{L_tv}\|_{L_2(\bR^d)} \leq N \||\xi|^\sigma \widehat{v}\|_{L_2(\bR^d)} = N\|(-\Delta)^{\sigma/2}u\|_{L_2(\bR^d)},
  \end{align*}
  which implies the continuity of $L_t$.
  The continuity of $\cL_t$ follows in the same way.
  For the special case where $\sigma=1$ and $L_tv=\nabla v$, one just needs to use $\widehat{\nabla v}(\xi) = -i\xi \widehat{v}(\xi)$ in place of \eqref{eq12041535}.

  For \eqref{oper_bdd}, by the Plancherel theorem and \eqref{ineq9260032},
    \begin{align*}
     \||\xi|^\sigma \widehat{v}\|_{L_2(\bR^d)} &\leq N\|-\mathrm{Re}(m(t,\xi))\widehat{v}\|_{L_2(\bR^d)} 
     \\
     &\leq N \|\widehat{L_tv}\|_{L_2(\bR^d)} = N \|L_tv\|_{L_2(\bR^d)}.
  \end{align*}
  This and the continuity of $\cL_t$ yield the desired result.

  $(ii)$
  We next consider the a priori estimate \eqref{mainest}. Due to \eqref{oper_bdd} with $L_t=-(-\Delta)^{\sigma/2}$, we only need to prove the estimate when $\cL_t=-(-\Delta)^{\sigma/2}$. In addition, due to the denseness of $C_c^\infty([0,T]\times\bR^d)$ in $\bH_{2,0}^{\sigma}(T)$ and the continuity of $L_t$, it suffices to prove the estimate for $u\in C_c^\infty([0,T]\times\bR^d)$ such that $u(0,\cdot)=0$.
  
  Let us multiply both sides of \eqref{maineq} by $(-\Delta)^{\sigma/2}u$. Then,
  \begin{align} \label{eq10171751}
    \int_{\bR^d_T} \partial_tu(-\Delta)^{\sigma/2}u \, dxdt &= \int_{\bR^d_T} L_tu(-\Delta)^{\sigma/2}u \, dxdt - \lambda \int_{\bR^d_T} u(-\Delta)^{\sigma/2}u \, dxdt \nonumber
    \\
    &\quad + \int_{\bR^d_T} f(-\Delta)^{\sigma/2}u \, dxdt.
  \end{align}
For the first term, since $u$ is a real function, by the Plancherel theorem,
\begin{align} \label{eq10171755}
  \int_{\bR^d_T} \partial_tu(-\Delta)^{\sigma/2}u \, dxdt &= \frac{1}{2}\int_{\bR^d_T} \overline{\widehat{\partial_tu}(t,\xi)}|\xi|^\sigma \widehat{u}(t,\xi) \, d\xi dt \nonumber
  \\
  &\quad + \frac{1}{2}\int_{\bR^d_T} \widehat{\partial_tu}(t,\xi) |\xi|^\sigma\overline{\widehat{u}(t,\xi)}  \, d\xi dt \nonumber
  \\
  &= \frac{1}{2} \int_{\bR^d_T} |\xi|^\sigma \partial_t(|\widehat{u}(t,\xi)|^2) \, d\xi dt = \frac{1}{2} \int_{\bR^d} |\xi|^\sigma |\widehat{u}(T,\xi)|^2 d\xi\geq0.
\end{align}
Again by the Plancherel theorem,
\begin{align}
  \int_{\bR^d_T} L_tu(-\Delta)^{\sigma/2}u \, dxdt &= \int_{\bR^d_T} m(t,\xi)|\xi|^\sigma |\widehat{u}(t,\xi)|^2 d\xi dt \nonumber
  \\
  &= \int_{\bR^d_T} \text{Re}(m(t,\xi))|\xi|^\sigma |\widehat{u}(t,\xi)|^2 d\xi dt \nonumber
  \\
  &\leq -N\int_{\bR^d_T} |\xi|^{2\sigma} |\widehat{u}(t,\xi)|^2 d\xi dt \nonumber
  \\
  &= -N\int_{\bR^d_T}|(-\Delta)^{\sigma/2}u(t,x)|^2 dxdt.
\end{align}
Here, for the inequality, we used \eqref{ineq9260032}.
For the second term on the right-hand side of \eqref{eq10171751}, the Plancherel theorem yields that
\begin{align}
    \lambda \int_{\bR^d_T} u(-\Delta)^{\sigma/2}u \, dxdt = \lambda \int_{\bR^d_T} |\xi|^\sigma |\widehat{u}(t,\xi)|^2 d\xi dt \geq 0.
\end{align}
For the last term, by H\"older's inequality, we get
\begin{align} \label{eq10171758}
    \int_{\bR^d_T} f(-\Delta)^{\sigma/2}u \, dxdt \leq \|(-\Delta)^{\sigma/2}u\|_{L_{2}(\bR_T^d)}\|f\|_{L_{2}(\bR_T^d)}.
\end{align}
Combining \eqref{eq10171755}-\eqref{eq10171758}, we have
\begin{equation} \label{eq10171759}
    \|(-\Delta)^{\sigma/2}u\|_{L_{2}(\bR_T^d)} \leq N\|f\|_{L_{2}(\bR_T^d)}.
\end{equation}

Similarly, by multiplying both sides of \eqref{maineq} by $\lambda u$ in place of $(-\Delta)^{\sigma/2} u$, one can deduce that
\begin{equation*}
    \lambda\|u\|_{L_{2}(\bR_T^d)} \leq N\|f\|_{L_{2}(\bR_T^d)}.
\end{equation*}
This, \eqref{eq10171759}, and the fact that $\partial_tu = L_tu-\lambda u+f$ lead to \eqref{mainest}.

$(iii)$ For the solvability, one just needs to use the method of continuity with \eqref{mainest}. We remark that the case $L_t=-(-\Delta)^{\sigma/2}$ was already proved (see e.g. \cite[Theorem 1]{MP19}).
The proposition is proved.
\end{proof}

\subsection{Auxiliary results}

In this section, we provide estimates necessary to derive $L_p$ estimates for all $p\in(1,\infty)$.

To handle general L\'evy measures, we define
\begin{equation} \label{eq10181920}
    (\cT_\kappa^R(t) u)(x) := \cT_\kappa^R u(t,x) := \kappa^\sigma R^{\sigma-d/p} \int_{B_{\kappa R}^c} \|u(t,\cdot)\|_{L_p(B_R(x+y))} \nu_t(dy),
\end{equation}
which can be interpreted as a convolution of $\|u(t,\cdot)\|_{L_p(B_R(y))}$ with $\nu_t(dy)$ over the tail.

\begin{lemma} \label{lem10142050}
  Let $p\in(1,\infty), \sigma\in(0,2), T\in(0,\infty), \lambda\geq0$, and $R>0$. Take $\zeta_0\in C_c^\infty(B_R)$ such that $\zeta_0=1$ in $B_{R/2}$ and
  \begin{align*}
    |D_x\zeta_0|\leq N R^{-1}, \quad |D_x^2\zeta_0|\leq N R^{-2}.
  \end{align*}
  Suppose that $L_t$ and $\cL_t$ satisfy Assumptions \ref{levy} and \ref{upper}, respectively. Assume that Theorem \ref{mainthm} holds for this $p$. If $u\in \bH_{p,0}^{\sigma}(T)$ satisfies
  \begin{equation*}
    \partial_tu = L_tu -\lambda u + f \text{ in } \bR_T^d,
  \end{equation*}
  then for any $\varepsilon \in (0,2^{-\sigma})$,
  \begin{align} \label{eq10120058}
  &\|\partial_t(u\zeta_0)\|_{L_p(\bR^d_T)} + \|\cL_t(u\zeta_0)\|_{L_p(\bR^d_T)} + \lambda \|u\zeta_0\|_{L_p(\bR^d_T)} \nonumber
    \\
    &\leq N_\varepsilon \left( \|f\|_{L_p\left((0,T)\times B_{R}\right)} + \frac{1}{R^{\sigma}} \|u\|_{L_p\left((0,T)\times B_{R}\right)} \right) \nonumber
    \\
    &\quad + \frac{N}{R^{\sigma-d/p}} \sum_{k=0}^\infty \varepsilon^k 2^{k\sigma} \left( \int_0^T \left|\cT_{p,2^{-k-4}}^R u(t,0) \right|^p dt \right)^{1/p},
  \end{align}
  where $N_\varepsilon=(\varepsilon,d,p,\sigma,\Lambda,N_0)$ and $N=N(d,p,\sigma,\Lambda,N_0)$. In particular, when $\sigma=1$, one can also consider $\nabla u$ instead of $\cL_t u$.
\end{lemma}

\begin{proof}
    Let 
    \begin{equation*}
         r_k:= R(1-2^{-k-1}), \quad k\geq0,
    \end{equation*}
    and consider cutoff functions $\zeta_k=\zeta_k(x)$ such that $\zeta_k \in C_c^\infty(B_{r_{k+1}})$, $0\leq\zeta_k\leq1$, $\zeta_k=1$ in $B_{r_k}$, and
  \begin{align*}
    |D_x\zeta_k|\leq N\frac{2^k}{R}, \quad |D_x^2\zeta_k|\leq N \frac{2^{2k}}{R^2}.
  \end{align*}

One can see that $u\zeta_k\in \bH_{p,0}^\sigma(T)$ satisfies
\begin{equation} \label{eq10112337}
    \partial_t (u\zeta_k) = L_t(u\zeta_k) -\lambda u\zeta_k + f\zeta_k + \zeta_kL_tu - L_t(u\zeta_k) \text{ in } \bR^d_T.
\end{equation}
Thus, if we apply Theorem \ref{mainthm} to \eqref{eq10112337}, then we have
\begin{align} \label{eq10120052}
&\|\partial_t(u\zeta_k)\|_{L_p(\bR^d_T)} + \|\cL_t(u\zeta_k)\|_{L_p(\bR^d_T)} + \lambda \|u\zeta_k\|_{L_p(\bR^d_T)} \nonumber
\\
&\leq N \left(\|f\zeta_k\|_{L_p(\bR^d_T)} + \|L_t(u\zeta_k)-\zeta_kL_tu\|_{L_p(\bR^d_T)} \right) \nonumber
\\
&\leq N\left(\|f\|_{L_p((0,T)\times B_R)} + \|L_t(u\zeta_k)-\zeta_kL_tu\|_{L_p(\bR^d_T)} \right). 
\end{align}

Now we consider three cases: $\sigma\in(0,1)$, $\sigma\in(1,2)$, and $\sigma=1$.

\vspace{1mm}
$(1)$ $\sigma\in(0,1)$.

In this case, \eqref{eq10120058} follows from \eqref{eq10120052} and \eqref{eq10102353} with $k=0$.

\vspace{1mm}
$(2)$ $\sigma\in(1,2)$.

In this case, by applying \eqref{eq10121035} to \eqref{eq10120052},
\begin{align} \label{eq10121053}
&\|\partial_t(u\zeta_k)\|_{L_p(\bR^d_T)} + \|\cL_t(u\zeta_k)\|_{L_p(\bR^d_T)} + \lambda \|u\zeta_k\|_{L_p(\bR^d_T)} \nonumber
\\
&\leq N \|f\|_{L_p((0,T)\times B_R)} + N\frac{2^{k(\sigma-1)}}{R^{\sigma-1}} \|\nabla u\|_{L_p\left((0,T)\times B_{r_{k+3}}\right)} + N \frac{2^{k\sigma}}{R^\sigma} \|u\|_{L_p\left((0,T)\times B_{R}\right)} \nonumber
      \\
      &\quad + N\frac{2^{k\sigma}}{R^{\sigma-d/p}} \left( \int_0^T \left|\cT_{p,2^{-k-4}}^R u(t,0) \right|^p dt \right)^{1/p}
\end{align}
By a well-known interpolation inequality (see e.g. \cite[Lemma A.1]{DL23}), for any $\varepsilon>0$,
\begin{equation*}
    \|\nabla v\|_{L_p(\bR^d)} \leq N(d,\sigma,p)\varepsilon^{1/(1-\sigma)}\|v\|_{L_p(\bR^d)} + \varepsilon \|(-\Delta)^{\sigma/2}v\|_{L_p(\bR^d)}.
\end{equation*}
Since Theorem \ref{mainthm} holds for $p$, by \eqref{oper_bdd} and choosing $\varepsilon>0$ appropriately again,
\begin{equation*}
    \|\nabla v\|_{L_p(\bR^d)} \leq N(d,\sigma,p,\Lambda,N_0)\varepsilon^{1/(1-\sigma)}\|v\|_{L_p(\bR^d)} + \varepsilon \|\cL_t v\|_{L_p(\bR^d)}
\end{equation*}
for any $t\in(0,T)$. Using this,
\begin{align*}
    &\frac{2^{k(\sigma-1)}}{R^{\sigma-1}} \|\nabla u\|_{L_p\left((0,T)\times B_{r_{k+3}}\right)} \leq \frac{2^{k(\sigma-1)}}{R^{\sigma-1}} \|\nabla (u\zeta_{k+3})\|_{L_p\left(\bR^d_T\right)} \nonumber
    \\
    &\leq N\frac{2^{k\sigma}}{R^\sigma}\varepsilon^{3/(1-\sigma)}\|u\zeta_{k+3}\|_{L_p(\bR^d_T)} + \varepsilon^3 \|\cL_t(u\zeta_{k+3})\|_{L_p(\bR^d_T)} \nonumber
    \\
    &\leq N\frac{2^{k\sigma}}{R^\sigma}\varepsilon^{3/(1-\sigma)}\|u\|_{L_p((0,T)\times B_R)} + \varepsilon^3 \|\cL_t(u\zeta_{k+3})\|_{L_p(\bR^d_T)}.
\end{align*}
Thus, by multiplying $\varepsilon^k$ to both sides of \eqref{eq10121053} and summing in $k$, we have
\begin{align} \label{eq10130033}
    &\sum_{k=0}^\infty \varepsilon^k \left(\|\partial_t(u\zeta_k)\|_{L_p(\bR^d_T)} + \lambda \|u\zeta_k\|_{L_p(\bR^d_T)}  +  \|\cL_t(u\zeta_k)\|_{L_p(\bR^d_T)} \right) \nonumber
    \\
    &\leq N \|f\|_{L_p((0,T)\times B_R)}\sum_{k=0}^\infty \varepsilon^k + N\frac{\varepsilon^{3/(1-\sigma)}}{R^\sigma}\|u\|_{L_p\left((0,T)\times B_{R}\right)} \sum_{k=0}^\infty(\varepsilon2^\sigma)^k \nonumber
    \\
    &\quad+ \sum_{k=0}^\infty \varepsilon^{k+3} \|\cL_t(u\zeta_{k+3})\|_{L_p(\bR^d_T)} \nonumber
    \\
    &\quad+ \frac{N}{R^{\sigma-d/p}} \sum_{k=0}^\infty \varepsilon^k 2^{k\sigma} \left( \int_0^T \left|\cT_{p,2^{-k-4}}^R u(t,0) \right|^p dt \right)^{1/p}.
\end{align}
To conclude \eqref{eq10120058}, it remains to prove $\sum_{k=0}^\infty \varepsilon^{k+3} \|\cL(u\zeta_{k+3})\|_{L_p(\bR^d_T)}<\infty$, which allows us to absorb this term in \eqref{eq10130033} to the left-hand side. Suppose that the right-hand side of \eqref{eq10120058} is finite.
Now we restrict the range of $\varepsilon\in(0,2^{-\sigma})$ to guarantee that $\sum_{k=0}^\infty(\varepsilon2^\sigma)^k < \infty$. In this range with $\sigma>1$, we further obtain that
\begin{equation*}
    \sum_{k=0}^\infty \varepsilon^k \frac{2^{k(\sigma-1)}}{R^{\sigma-1}} \|\nabla u\|_{L_p\left((0,T)\times B_{r_{k+3}}\right)} \leq N \sum_{k=0}^\infty \varepsilon^k \frac{2^{k(\sigma-1)}}{R^{\sigma-1}} \|u\|_{\bH_p^\sigma(T)} <\infty.
\end{equation*}
Thus, if we multiply both sides of \eqref{eq10121053} by $\varepsilon^k$ and take the sum over $k$, then we see that
\begin{equation} \label{eq10132057}
    \sum_{k=0}^\infty \varepsilon^k \|\cL_t(u\zeta_k)\|_{L_p(\bR^d_T)} < \infty.
\end{equation}
Hence, the claim is proved.

\vspace{1mm}
  $(3)$ $\sigma=1$.

In this case, thanks to \eqref{oper_bdd}, it suffices to prove \eqref{eq10120058} with $L_t (u\zeta_0)$ instead of $\cL_t (u\zeta_0)$ on the left-hand side.
  As in \eqref{eq10121053}, by \eqref{eq10120052} and \eqref{eq10131542}, 
\begin{align} \label{eq10132056}
&\|\partial_t(u\zeta_k)\|_{L_p(\bR^d_T)} + \|L_t(u\zeta_k)\|_{L_p(\bR^d_T)} + \lambda \|u\zeta_k\|_{L_p(\bR^d_T)} \nonumber
\\
&\leq N \|f\|_{L_p((0,T)\times B_R)} + N\varepsilon^{-3}\frac{2^k}{R} \|u\|_{L_p\left((0,T)\times B_{R}\right)} \nonumber
      \\
      &\quad + \varepsilon^3\|\nabla(u\zeta_{k+3})\|_{L_p\left(\bR^d_T\right)} + N \frac{2^{k}}{R^{1-d/p}} \left( \int_0^T \left|\cT_{p,2^{-k-4}}^R u(t,0) \right|^p dt \right)^{1/p}.
\end{align}
Since $u\in \bH_p^1(T)$, for $\varepsilon<2^{-\sigma}=2^{-1}$,
\begin{equation*}
    \sum_{k=0}^\infty \varepsilon^{k+3}\|\nabla(u\zeta_{k+3})\|_{L_p\left(\bR^d_T\right)} \leq N(R)\sum_{k=0}^\infty \varepsilon^{k+3} \left(\|\nabla u\|_{L_p\left(\bR^d_T\right)} + 2^k\|u\|_{L_p\left(\bR^d_T\right)} \right)<\infty.
\end{equation*}
Hence, this and \eqref{eq10132056} lead to \eqref{eq10132057}.
  Note that by \eqref{oper_bdd} with
  $\nabla$ instead of $\cL_t$,
  \begin{align*}
      \|\nabla u\|_{L_p\left((0,T)\times B_{r_{k+3}}\right)} \leq \|\nabla (u\zeta_{k+3})\|_{L_p\left(\bR^d_T\right)} \leq N\|L_t (u\zeta_{k+3})\|_{L_p\left(\bR^d_T\right)},
  \end{align*}
where we used the nondegenerate condition \eqref{nonde}.
Thus, by \eqref{eq10132056} and replacing $\varepsilon>0$ appropriately again, 
\begin{align*}
&\|\partial_t(u\zeta_k)\|_{L_p(\bR^d_T)} + \|L_t(u\zeta_k)\|_{L_p(\bR^d_T)} + \lambda \|u\zeta_k\|_{L_p(\bR^d_T)} \nonumber
\\
&\leq N \|f\|_{L_p((0,T)\times B_R)} + N\varepsilon^{-3}\frac{2^k}{R} \|u\|_{L_p\left((0,T)\times B_{R}\right)} \nonumber
      \\
      &\quad + \varepsilon^3\|L_t(u\zeta_{k+3})\|_{L_p\left(\bR^d_T\right)} + N \frac{2^{k}}{R^{1-d/p}} \left( \int_0^T \left|(\cT_{p,2^{-k-4}}^R u)(t,0) \right|^p dt \right)^{1/p}.
\end{align*}
Similar to \eqref{eq10130033}, we can conclude \eqref{eq10120058} by multiplying both sides by $\varepsilon^k$ and taking the sum over $k$.
  The lemma is proved.
\end{proof}

\begin{remark} \label{rem1127}
    In Lemma \ref{lem10142050}, we can replace the condition $u\in\bH_{p,0}^\sigma(T)$ with the weaker condition
    \begin{equation*}
        u\zeta_k \in \bH_{p,0}^\sigma(T), \quad L_tu \in L_{p}((0,T)\times B_{r_{k+1}})  \quad \text{ for any } k\in \bN,
    \end{equation*}
    where $\zeta_k$ is defined as in the proof of the lemma.
\end{remark}

Let us introduce some notation which will be used below. 
For a given nonlocal operator $\cL_t$, we denote by $\widetilde{\nu}_t$ its corresponding L\'evy measure, and
  \begin{equation*}
      \widetilde{\cT}_{p,\kappa}^R u(t,x) := \kappa^\sigma R^{\sigma-d/p} \int_{B_{\kappa R}^c} \|u(t,\cdot)\|_{L_p(B_R(x+y))} \widetilde{\nu}_t(dy).
  \end{equation*}
For $R>0$ and $(t,x)\in \bR^{d+1}$, we write
\begin{equation*}
    Q_R(t,x)=(t-R^\sigma,t)\times B_R(x).
\end{equation*}
For a function $g$ on $\bR^{d+1}$ and any $D\subset \bR^{d+1}$, we set
\begin{equation*}
    (g)_D=\aint_D |g(t,x)|dtdx,
\end{equation*}
and the (parabolic) maximal function
\begin{equation*}
    \bM g(t_0,x_0) = \sup_{Q_R(t,x)\ni (t_0,x_0)} (g)_{Q_R(t,x)}.
\end{equation*}
We also consider the standard maximal function in the time variable
\begin{equation} \label{eq11192353}
    \bM_t h(t_0) = \sup_{(t-r^\sigma,t)\ni t_0} \aint_{t-r^\sigma}^{t} h(s)ds,
\end{equation}
and the following maximal operators involving L\'evy measure
\begin{align*}
      (\bT_\kappa(t) u)(x) := \bT_\kappa u(t,x) = \sup_{R>0} \kappa^\sigma R^{\sigma-d} \int_{B_{\kappa R}^c} \int_{B_{R}} |u(x+y+z)| dz\nu_t(dy).
  \end{align*}
Similar to $\widetilde{\cT}_{p,\kappa}^R u$, $(\widetilde{\bT}_\kappa(t) u)(x) = \widetilde{\bT}_\kappa u(t,x)$ is defined by using $\widetilde{\nu}_t(dy)$ in place of $\nu_t(dy)$.

\begin{lemma} \label{lem10151525}
    Let $p\in(1,\infty), \sigma\in(0,2), T\in(0,\infty)$, and $\lambda\geq0$.
  Suppose that $L_t$ and $\cL_t$ satisfy Assumptions \ref{levy} and \ref{upper}, respectively, and that $\cL_t\equiv\cL$ is independent of $t$. Assume that Theorem \ref{mainthm} holds for this $p$. Let $u\in \bH_{p,0}^{\sigma}(T)$ satisfy
  \begin{equation*}
    \partial_tu = L_tu -\lambda u + f \text{ in } \bR_T^d.
  \end{equation*}
  Then there exists $p_1=p_1(d,\sigma,p)\in(p,\infty]$ such that
  \begin{equation} \label{eq10201608}
      p_1-p>\delta(d,\sigma)>0,
  \end{equation}
  and the following holds. For any $(t_0,x_0)\in \bR^d_T$ and $R>0$, there exist $v,w\in \bH_{p,0}^\sigma(S,t_0)$ with $S:=\min\{0,t_0-R^\sigma\}$ such that $u=v+w$ and
  \begin{align} \label{eq101032344}
      (|\cL w|^p)_{Q_{R}(t_0,x_0)}^{1/p} + (|\lambda w|^p)_{Q_{R}(t_0,x_0)}^{1/p} \leq N(|f|^p)_{Q_{2R}(t_0,x_0)}^{1/p},
  \end{align}
  \begin{align} \label{eq101032342}
      &(|\cL v|^{p_1})_{Q_{R/2}(t_0,x_0)}^{1/p_1} \nonumber
      \\
      &\leq N (|f|^{p})_{Q_{2R}(t_0,x_0)}^{1/p} + N\left(\aint_{t_0-R^\sigma}^{t_0} \left|\widetilde{\cT}_{p,2^{-1}}^Rf(t,x_0)\right|^{p} dt\right)^{1/p} + N (|\cL u|^{p})_{Q_{R}(t_0,x_0)}^{1/p} \nonumber
      \\
      &\quad+ N\sum_{k=0}^\infty 2^{-k\sigma} \left(\aint_{t_0-R^\sigma}^{t_0} \left|\cT^R_{p,2^{-k-4}} (\cL u)(t,x_0)\right|^p dt\right)^{1/p},
  \end{align}
  and
  \begin{align} \label{eq101032343}
      &(|\lambda v|^{p_1})_{Q_{R/2}(t_0,x_0)}^{1/{p_1}} \nonumber
      \\
      &\leq N (|f|^{p})_{Q_{2R}(t_0,x_0)}^{1/p} + N \left( \aint_{t_0-R^\sigma}^{t_0} \left|\widetilde{\cT}_{p,2^{-1}}^Rf(t,x_0) \right|^p dt \right)^{1/p} + N (|\lambda u|^{p})_{Q_{R}(t_0,x_0)}^{1/p} \nonumber
      \\
      &\quad+ N\sum_{k=0}^\infty 2^{-k\sigma} \left(\aint_{t_0-R^\sigma}^{t_0} \left|\cT^R_{p,2^{-k-4}} (\lambda u) (t,x_0) \right|^p dt\right)^{1/p},
  \end{align}
  where $N=N(d,p,\sigma,\Lambda,N_0)$, and $u$ and $f$ are extended to be zero for $t<0$.
\end{lemma}

\begin{proof}
Without loss of generality, we may assume that $x_0=0$. We also recall that the extended function $u$, which is in the space $\bH_{p,0}^\sigma(S,T)$, is a solution to 
  \begin{equation*}
    \partial_tu = L_tu -\lambda u + f \quad \text{ in } (S,T)\times \bR^d.
  \end{equation*}

First, we prove \eqref{eq101032344}.
    Let $\eta$ be a cutoff function such that $0\leq \eta \leq 1$, $\text{supp}(\eta)\subset (t_0-(2R)^{\sigma}, t_0+(2R)^{\sigma})\times B_{2R}$, and $\eta=1$ in $Q_{3R/2}(t_0,0)$. By Theorem \ref{mainthm}, there exists $w \in \bH_{p,0}^\sigma(S,t_0)$ such that 
  \begin{equation*}
    \partial_tw = L_tw -\lambda w + \eta f \text{ in } (S,t_0)\times \bR^d.
  \end{equation*}
  Moreover, due to \eqref{mainest},
    \begin{eqnarray} \label{eq10140009}
         &\|\partial_t w\|_{L_{p}((S,t_0)\times \bR^d)} + \|\cL w\|_{L_{p}((S,t_0)\times \bR^d)} + \lambda\|w\|_{L_{p}((S,t_0)\times \bR^d)}& \nonumber
         \\
         &\leq N\|\eta f\|_{L_{p}((S,t_0)\times \bR^d)} \leq N\|f\|_{L_{p}(Q_{2R}(t_0,0))}.&
    \end{eqnarray}
    Here, we emphasize that the theorem can be applied with general $S\neq0$ by a shift of the coordinates.
    Now \eqref{eq101032344} can be obtained by dividing both sides of \eqref{eq10140009} by $R^{(d+\sigma)/p}$.

    Next, we show \eqref{eq101032342}. Let $v:=u-w$, and we further take $\chi\in C^\infty(\bR)$ so that $0\leq\chi\leq1$, $\text{supp}(\chi)\subset (t_0-R^\sigma,t_0+R^\sigma)$, and $\chi=1$ in $(t_0-(R/2)^\sigma,t_0)$. Then $\psi(t,x):=v(t,x)\chi(t) \in \bH_{p,0}^\sigma(S,t_0)$ satisfies
    \begin{equation} \label{eq10142001}
    \partial_t\psi = L_t\psi -\lambda \psi + \chi(1-\eta) f + v\partial_t\chi \text{ in } (S,t_0)\times\bR^d.
  \end{equation}
Since $\cL$ is time-independent, one can take $\cL$ on both sides to obtain that
\begin{equation*}
        \partial_t (\cL\psi) = L(\cL\psi) -\lambda \cL\psi + \cL[\chi(1-\eta) f] + (\cL v)\partial_t\chi \text{ in }(S,t_0)\times\bR^d.
\end{equation*}
If necessary, we mollify \eqref{eq10142001} to justify this procedure.

Let $\zeta_0$ be taken from Lemma \ref{lem10142050}, and $p_1\in (p,\infty)$ so that
\begin{equation*}
    1/p_1=1/p-\sigma/2(d+\sigma) \quad \text{ if } p\leq d/\sigma +1,
\end{equation*}
and $p_1=\infty$ if $p>d/\sigma+1$. Then due to $p>1$, when $p_1<\infty$,
\begin{equation*}
    p_1-p = \frac{2p(d+\sigma)}{2d+2\sigma-p\sigma}-p = \frac{p^2\sigma}{2d+2\sigma-p\sigma} > \frac{\sigma}{2d+\sigma},
\end{equation*}
which guarantees \eqref{eq10201608}. We also note that this is obvious when $p_1=\infty$.

  By \eqref{eq11072241} with $p_1$ instead of $l$, and \eqref{eq10120058} with $
  \varepsilon=2^{-2\sigma}$,
  \begin{align} \label{eq10151503}
      &(|\cL v|^{p_1})_{Q_{R/2}(t_0,0)}^{1/{p_1}} = (|\cL \psi|^{p_1})_{Q_{R/2}(t_0,0)}^{1/{p_1}} \nonumber
      \\
      &\leq NR^\vartheta \left( \|\partial_t(\zeta_0\cL \psi)\|_{L_p((t_0-R^\sigma,t_0)\times \bR^d)} + \|(-\Delta)^{\sigma/2}(\zeta_0\cL \psi)\|_{L_p((t_0-R^\sigma,t_0)\times \bR^d)} \right) \nonumber
      \\
            &\leq NR^\vartheta \left( \|\partial_t(\zeta_0\cL \psi)\|_{L_p((S,t_0)\times \bR^d)} + \|(-\Delta)^{\sigma/2}(\zeta_0\cL \psi)\|_{L_p((S,t_0)\times \bR^d)} \right) \nonumber
      \\
        &\leq N R^\vartheta \|\cL[\chi(1-\eta) f]\|_{L_p\left((S,t_0)\times B_{R}\right)} \nonumber
          \\
        &\quad+ N R^\vartheta \|(\cL v)\partial_t\chi\|_{L_p\left((S,t_0)\times B_{R}\right)} + N R^{-(d+\sigma)/p} \|\cL\psi\|_{L_p\left((S,t_0)\times B_{R}\right)}\nonumber
    \\
    &\quad + NR^{-\sigma/p}\sum_{k=0}^\infty  2^{-k\sigma} \left(\int_{S}^{t_0} \left|\cT^R_{p,2^{-k-4}}(\cL \psi) (t,0) \right|^p dt\right)^{1/p},
  \end{align}
  where $\vartheta:=\sigma-(d+\sigma)/p$.
  Here, as in \eqref{eq10140009}, \eqref{eq10120058} is applied for a general time interval $(S,t_0)$.
  Since $|\partial_t\chi|\leq NR^{-\sigma}$, $\text{supp}(\chi)\subset (t_0-R^\sigma,t_0+R^\sigma)$, and $v=u-w$,
  \begin{align} \label{eq10151506}
       &R^\vartheta \|(\cL v)\partial_t\chi\|_{L_p\left((S,t_0)\times B_{R}\right)} + R^{-(d+\sigma)/p} \|\cL\psi\|_{L_p\left((S,t_0)\times B_{R}\right)} \nonumber
       \\
       &\leq N (|\cL v|^{p})_{Q_{R}(t_0,0)}^{1/p} \leq N \left( (|\cL u|^{p})_{Q_{R}(t_0,0)}^{1/p} +  (|\cL w|^{p})_{Q_{R}(t_0,0)}^{1/p} \right) \nonumber
       \\
       &\leq N \left( (|\cL u|^{p})_{Q_{R}(t_0,0)}^{1/p} +  (|f|^{p})_{Q_{2R}(t_0,0)}^{1/p} \right).
  \end{align}
  Here, the last inequality follows from \eqref{eq10140009}.
  For the last term of \eqref{eq10151503}, by using $\text{supp}(\chi)\subset (t_0-R^\sigma,t_0+R^\sigma)$ and $v=u-w$,
  \begin{align} \label{eq10151504}
      &\int_{t_0-R^\sigma}^{t_0} |\cT^R_{p,2^{-k-4}}(\cL \psi)(t,0)|^p dt \leq N \int_{t_0-R^\sigma}^{t_0} |\cT^R_{p,2^{-k-4}}(\cL v) (t,0)|^p dt \nonumber
      \\
      &\leq N \int_{t_0-R^\sigma}^{t_0} |\cT^R_{p,2^{-k-4}}(\cL u) (t,0)|^p dt + N \int_{t_0-R^\sigma}^{t_0} |\cT^R_{p,2^{-k-4}}(\cL w) (t,0)|^p dt.
  \end{align}
Due to \eqref{eq10181920}, \eqref{equiv}, and \eqref{eq10140009},
\begin{align} \label{eq10192338}
    &\int_{t_0-R^\sigma}^{t_0} |\cT^R_{p,2^{-k-4}}(\cL w)|^p (t,0) dt \nonumber
    \\
    &\leq N \int_{t_0-R^\sigma}^{t_0} \left| 2^{-k\sigma} R^{\sigma-d/p} \|\cL w(t,\cdot)\|_{L_p(\bR^d)}^p \int_{B_{2^{-k-4} R}^c} \nu_t(dy) \right|^{p} dt \nonumber
    \\
    &\leq N R^{-d} \|\cL w\|_{L_p((t_0-R^\sigma,t_0)\times \bR^d)}^p \nonumber
    \\
    &\leq NR^{-d} \|f\|_{L_{p}(Q_{2R}(t_0,0))}^p = N R^\sigma (|f|^{p})_{Q_{2R}(t_0,0)}.
\end{align}
Hence, to conclude \eqref{eq101032342}, it remains to estimate $\|\cL[\chi(1-\eta) f]\|_{L_p\left((S,t_0)\times B_{R}\right)}$. Let $\widetilde{\nu}_t$ be the corresponding L\'evy measure of $\cL$. Since $1-\eta=0$ in $Q_{3R/2}(t_0,0)$, for $x\in B_R$,
  \begin{align*}
      \left|\cL[\chi(1-\eta) f](t,x)\right| &= \chi(t) \left|\int_{|x+y|>3R/2} (1-\eta)(t,x+y)f(t,x+y) \widetilde{\nu}_t(dy)\right|
      \\
      &\leq \chi(t) \int_{|y|>R/2} |f(t,x+y)| \widetilde{\nu}_t(dy).
  \end{align*}
Thus, by the Minkowski inequality, H\"older's inequality, and \eqref{upper_mu},
\begin{align*}
    \|\cL[\chi(1-\eta) f](t)\|_{L_p\left(B_R\right)} \leq N R^{-\sigma+d/p} \chi(t) (\widetilde{\cT}_{p,2^{-1}}^Rf)(t,0),
\end{align*}
which easily yields that
\begin{align} \label{eq10151505}
    R^\vartheta \|\cL[\chi(1-\eta) f]\|_{L_p\left((S,t_0)\times B_{R}\right)} \leq N \left(\aint_{t_0-R^\sigma}^{t_0} \left|\widetilde{\cT}_{p,2^{-1}}^R f(t,0) \right|^p dt\right)^{1/p}.
\end{align}
Combining \eqref{eq10151503}, \eqref{eq10151506}, \eqref{eq10151504}, \eqref{eq10192338}, and \eqref{eq10151505}, we obtain \eqref{eq101032342}.

Finally, for \eqref{eq101032343}, one just needs to repeat the above argument. Here, we only emphasize that $\lambda \psi$ satisfies
\begin{equation*}
    \partial_t (\lambda\psi) = L_t(\lambda\psi) -\lambda (\lambda\psi) + \lambda[\chi(1-\eta) f] + \lambda v\partial_t\chi \text{ in } \bR_T^d,
\end{equation*}
and $\lambda[\chi(1-\eta) f]=0$ in $Q_{3R/2}(t_0,0)$.
The lemma is proved.  
\end{proof}

\begin{remark}
In the above proof, \eqref{eq10151503} is derived under the choice  $\varepsilon=2^{-2\sigma}$ for simplicity. The argument, however, remains valid for any $\varepsilon\in (0,2^{-\sigma})$, and this does not affect the validity of the proof of our main theorem. In particular, instead of \eqref{eq101032342}, we arrive at
\begin{align*}
     &(|\cL v|^{p_1})_{Q_{R/2}(t_0,x_0)}^{1/p_1} \nonumber
      \\
      &\leq N_\varepsilon (|f|^{p})_{Q_{2R}(t_0,x_0)}^{1/p} + N_\varepsilon \left(\aint_{t_0-R^\sigma}^{t_0} \left|\widetilde{\cT}_{p,2^{-1}}^Rf(t,x_0)\right|^{p} dt\right)^{1/p} + N (|\cL u|^{p})_{Q_{R}(t_0,x_0)}^{1/p} \nonumber
      \\
      &\quad+ N\sum_{k=0}^\infty \varepsilon^{k} 2^{k\sigma} \left(\aint_{t_0-R^\sigma}^{t_0} \left|\cT^R_{p,2^{-k-4}} (\cL u)(t,x_0)\right|^p dt\right)^{1/p}.
\end{align*}
\end{remark}

Let $p\in(1,\infty)$ and $p_1\in(p,\infty]$ be taken from Lemma \ref{lem10151525}. For $\gamma>0$ and $\rho>0$, denote
\begin{equation} \label{eq10161704}
    \cA(\rho) := \{(t,x)\in (-\infty,T)\times \bR^d: |\cL u(t,x)|> \rho\},
\end{equation}
\begin{equation*}
    \cA'(\rho) := \{(t,x)\in (-\infty,T)\times \bR^d: |\lambda u(t,x)|> \rho\},
\end{equation*}
\begin{align} \label{eq10161705}
    \cB_\gamma(\rho) :=&\{(t,x)\in (-\infty,T)\times \bR^d: \gamma^{-1/p}(\bM |f|^p(t,x))^{1/p} \nonumber
    \\
    &+ \gamma^{-1/p}(\bM_t \widetilde{\bT}_{2^{-1}}|f|^p(t,x))^{1/p} + \gamma^{-1/p_1}(\bM |\cL u|^p(t,x))^{1/p} \nonumber
    \\
    &+ \gamma^{-1/p_1}  \sum_{k=0}^\infty 2^{-k\sigma} \left(\bM_t \bT_{2^{-k-4}}(|\cL u|^p)(t,x)\right)^{1/p} > \rho\},
\end{align}
and
\begin{align*}
    \cB_\gamma'(\rho) :=  &\{(t,x)\in (-\infty,T)\times \bR^d: \gamma^{-1/p}(\bM |f|^p(t,x))^{1/p}
    \\
    &+ \gamma^{-1/p}(\bM_t \widetilde{\bT}_{2^{-1}}|f|^p(t,x))^{1/p} + \gamma^{-1/p_1}(\bM |\lambda u|^p(t,x))^{1/p} 
    \\
    &+ \gamma^{-1/p_1} \sum_{k=0}^\infty  2^{-k\sigma} \left(\bM_t \bT_{2^{-k-4}}(|\lambda u|^p)(t,x)\right)^{1/p} > \rho\}.
\end{align*}
In addition, we write
\begin{equation} \label{eq10152356}
    \cC_R(t,x) := (t-R^\sigma, t+R^\sigma)\times B_R(x), \quad \widehat{\cC_R}(t,x) := \cC_R(t,x)\cap \{t< T\}.
\end{equation}

\begin{lemma} \label{lem10152345}
    Let $\gamma\in(0,1)$, $p\in(1,\infty), \sigma\in(0,2), T\in(0,\infty), \lambda\geq0$, and $R>0$. 
  Suppose that $L_t$ and $\cL_t$ satisfy Assumptions \ref{levy} and \ref{upper}, respectively, and that $\cL_t \equiv \cL$ is independent of $t$. Assume that Theorem \ref{mainthm} holds for this $p$. Let $u\in \bH_{p,0}^{\sigma}(T)$ satisfy
  \begin{equation} \label{eq10151550}
    \partial_tu = L_tu -\lambda u + f \text{ in } \bR_T^d.
  \end{equation}
  Then there exists a sufficiently large constant $\kappa=\kappa(d,\Lambda,p,\sigma,N_0)>1$ so that for any $(t_0,x_0)\in \bR^d_T$ and $s>0$, if
  \begin{equation} \label{eq10152138}
      |\cC_{\widetilde{R}}(t_0,x_0)\cap \cA(\kappa \rho)| \geq \gamma |\cC_{\widetilde{R}}(t_0,x_0)|,
  \end{equation}
  then
  \begin{equation*}
      \widehat{\cC_{\widetilde{R}}}(t_0,x_0) \subset \cB_\gamma(\rho),
  \end{equation*}
      where $\widetilde{R}:=2^{-1-1/\sigma}R$. Moreover, the claim still holds if we replace $\cA(\kappa \rho)$ and $\cB(\rho)$ with $\cA'(\kappa \rho)$ and $\cB'(\rho)$, respectively.
\end{lemma}

\begin{proof}
We only consider $\cA(\kappa \rho)$ and $\cB(\rho)$ since the proof for $\cA'(\kappa \rho)$ and $\cB'(\rho)$ is similar.
    Moreover, it suffices to assume that $\rho=1$ since we can divide \eqref{eq10151550} by $\rho$.

    We will use a proof by contradiction. Suppose that there is $(s,y)\in \widehat{\cC_{\widetilde{R}}}(t_0,x_0)$ such that $(s,y)\notin \cB_\gamma(1)$, i.e., 
    \begin{align*}
    &\gamma^{-1/p}(\bM |f|^p(s,y))^{1/p} + \gamma^{-1/p}(\bM_t \widetilde{\bT}_{2^{-1}}|f|^p(s,y))^{1/p} + \gamma^{-1/p_1}(\bM |\cL u|^p(s,y))^{1/p} 
    \\
    &\quad+ \gamma^{-1/p_1} \sum_{k=0}^\infty 2^{-k\sigma} (\bM_t \bT_{2^{-k-4}}(|\cL u|^p)(s,y))^{1/p} \leq 1.
    \end{align*}
    Let $t_1:=\min\{t_0 + \widetilde{R}^\sigma,T\}$. Then by Lemma \ref{lem10151525} with $t_1$ instead of $t_0$, one can find $v,w\in\bH_{p,0}^\sigma(t_1-R^\sigma,t_1)$ satisfying 
      \begin{align*}
      (|\cL w|^p)_{Q_{R}(t_1,x_0)}^{1/p} \leq N(|f|^p)_{Q_{2R}(t_1,x_0)}^{1/p},
  \end{align*}
  and
  \begin{align*}
      (|\cL v|^{p_1})_{Q_{R/2}(t_1,x_0)}^{1/p_1}  &\leq N(|f|^p)_{Q_{2R}(t_1,x_0)}^{1/p} + N\left(\aint_{t_1-R^\sigma}^{t_1} \left|\widetilde{T}^R_{p,2^{-1}}f(t,x_0)\right|^p dt \right)^{1/p}
      \\
      &\quad+ N(|\cL u|^p)_{Q_{R}(t_1,x_0)}^{1/p} \nonumber
      \\
      &\quad+ N\sum_{k=0}^\infty 2^{-k\sigma} \left(\aint_{t_1-R^\sigma}^{t_1} \left|T^R_{2^{-k-4}}(\cL u) (t,x_0) \right|^p dt\right)^{1/p}.
  \end{align*}
  This, \eqref{eq11221518}, and the set inclusions
  \begin{align*}
      (s,y)\in \widehat{\cC_{\widetilde{R}}}(t_0,x_0) \subset Q_{R/2}(t_1,x_0) \subset Q_{2R}(t_1,x_0),
  \end{align*}
yield that
  \begin{equation*}
      (|\cL v|^{p_1})_{Q_{R/2}(t_1,x_0)}^{1/{p_1}} \leq N\gamma^{1/p_1} ,\quad (|\cL w|^p)_{Q_{R}(t_1,x_0)}^{1/p} \leq N\gamma^{1/p},
  \end{equation*}
  where $N=N(d,p,\Lambda,\sigma,N_0)$. Let $C_1\in(0,\kappa)$ be a constant, which will be determined later. By the Chebyshev inequality, 
  \begin{align*}
      |\cC_{\widetilde{R}}(t_0,x_0)\cap \cA(\kappa)| &\leq |\{(t,x)\in Q_{R/2}(t_1,x_0): |\cL u(t,x)|>\kappa\}|
      \\
      &\leq |\{(t,x)\in Q_{R/2}(t_1,x_0): |\cL w(t,x)|>\kappa-C_1\}|
      \\
      &\quad+ |\{(t,x)\in Q_{R/2}(t_1,x_0): |\cL v(t,x)|>C_1\}|
      \\
      &\leq \int_{Q_{R/2}(t_1,x_0)} (\kappa-C_1)^{-p}|\cL w|^p + C_1^{-p_1}|\cL v|^{p_1} dxdt
      \\
      &\leq N\left( \frac{\gamma|Q_{R/2}|}{(\kappa-C_1)^p} + 1_{p_1<\infty} \frac{\gamma|Q_{R/2}|}{C_1^{p_1}} \right)
      \\
      &\leq N\gamma|\cC_{\widetilde{R}}(t_0,x_0)|\left( (\kappa-C_1)^{-p} + 1_{p_1<\infty} C_1^{-p_1}\right).
  \end{align*}
  Thus, by choosing sufficiently large $C_1$ and $\kappa$ in order, we get
  \begin{equation*}
       |\cC_{\widetilde{R}}(t_0,x_0)\cap \cA(\kappa)| < \gamma |\cC_{\widetilde{R}}(t_0,x_0)|,
  \end{equation*}
  which contradicts \eqref{eq10152138}.
  The lemma is proved.
\end{proof}

\subsection{Proof of Theorem \ref{mainthm}}

\begin{proof}[Proof of Theorem \ref{mainthm}]

\textbf{1.} In \textbf{1} and \textbf{2}, we prove the theorem under the assumption that the continuity of $L_t$ is already established. In the present step, we prove $(i)$ and $(iii)$ assuming for the moment that \eqref{mainest} is derived. 
 For $v\in H_p^\sigma(\bR^d)$, we define $u(t,x):=\eta(t/T)v(x)$, where $\eta\in C_c^\infty([0,\infty))$ is a nonzero function so that $\eta(0)=0$. Our aim is to prove \eqref{oper_bdd} for a fixed $t=t_0>0$. By applying \eqref{mainest} with $\lambda=0$ and time-independent operators $\cL_{t_0}$ and $L_{t_0}$,
\begin{align*}
   \|\eta(\cdot/T)\|_{L_p((0,T))} \|\cL_{t_0} v\|_{L_p(\bR^d)} &= \|\cL_{t_0} u\|_{L_p(\bR^d_T)} 
   \\
   &\leq N \left(\|L_{t_0} u\|_{L_p(\bR^d_T)} +  T^{-1}\|(\partial_t\eta)(\cdot/T) v\|_{L_p(\bR^d_T)} \right)
    \\
    &= N \|\eta(\cdot/T)\|_{L_p((0,T))}\|L_{t_0} v\|_{L_p(\bR^d)} 
    \\
    &\quad + N T^{-1}\|(\partial_t\eta)(\cdot/T)\|_{L_p((0,T))}\|v\|_{L_p(\bR^d)},
\end{align*}
where $N$ is independent of $T$. 
Dividing both sides by $ \|\eta(\cdot/T)\|_{L_p((0,T))}$ and letting $T\to\infty$, we obtain the continuity of $\cL_{t_0}$ as well as \eqref{oper_bdd}.

Next, we prove $(iii)$ under the assumption that \eqref{mainest} holds, in particular, when $L_t=-(-\Delta)^{\sigma/2}$. 
Since
\begin{equation*}
    \|u\|_{H_{p}^\sigma(\bR^d)} \approx \|u\|_{L_{p}(\bR^d)} + \|(-\Delta)^{\sigma/2}u\|_{L_{p}(\bR^d)}
\end{equation*}
(see e.g. \cite[Lemma 3.4]{DJK23}), combining \eqref{mainest} with $\cL_t=-(-\Delta)^{\sigma/2}$ and \eqref{eq110722228} leads to
\begin{equation} \label{eq10152309}
    \|u\|_{\bH_{p}^\sigma(T)} \leq N (1+\min\{T,\lambda^{-1}\}) \|f\|_{L_{p}(\bR_T^d)},
\end{equation}
where $N$ is independent of $T$ and $\lambda$. The uniqueness easily follows by considering $f=0$ in \eqref{eq10152309}. For the existence, note that the case $L_t = -(-\Delta)^{\sigma/2}$ was already proved (see e.g. \cite[Theorem 1]{MP19}). This fact, together with \eqref{eq10152309} and the method of continuity, yields the desired result of $(iii)$.

\textbf{2.} Now we focus on $(ii)$. Due to the continuity of $L_t$ and the definition of $\bH_{p,0}^\sigma(T)$, we just need to prove \eqref{mainest} when $u\in C_c^\infty([0,T]\times\bR^d)$ with $u(0,x)=0$.
Let us deal with the case when $p\in[2,\infty)$. We first proceed with the proof under the additional assumption that $\cL_t\equiv \cL$ is independent of $t$. As in the proof of \cite[Theorem 2.6]{DL23}, we use an iterative argument to successively
increase the desired exponent $p$. First, the case when $p=2$ is proved in Proposition \ref{prop2}. Next, assume that the theorem holds for some $p_0\in[2,\infty)$. Take $p_1=p_1(d,\sigma,p_0)$ from Lemma \ref{lem10151525}. Our goal is to show \eqref{mainest} for $p\in(p_0,p_1)$. By Lemmas \ref{lem10152345} and \ref{lem10161553}, there is $\kappa=\kappa(d,p_0,\Lambda,\sigma,N_0)$ so that
\begin{equation*}
    |\cA(\kappa \rho)| \leq N(d)\gamma |\cB_\gamma(\rho)|,
\end{equation*}
where $\cA(\rho)$ and $\cB_\gamma(\rho)$ are defined in the same way as in \eqref{eq10161704} and \eqref{eq10161705}, except that $p$ is replaced by $p_0$. Thus,
\begin{align} \label{eq10161724}
    \|\cL u\|_{L_p(\bR^d_T)}^p = p\kappa^p \int_0^\infty |\cA(\kappa s)| s^{p-1} ds \leq N(d,p)\gamma \kappa^p \int_0^\infty |\cB_\gamma(s)| s^{p-1} ds.
\end{align}
Let us estimate the terms appearing in the definition of $\cB_\gamma(s)$. First, by the Hardy–Littlewood theorem,
\begin{align} \label{eq10161716}
    &\|\gamma^{-1/p_0}(\bM|f|^{p_0})^{1/p_0}\|_{L_p(\bR^d_T)} + \|\gamma^{-1/p_1}(\bM|
    \cL u|^{p_0})^{1/p_0}\|_{L_p(\bR^d_T)} \nonumber
    \\
    &\leq N \gamma^{-1/p_0}\|f\|_{L_p(\bR^d_T)} +
   N\gamma^{-1/p_1} \|\cL u\|_{L_p(\bR^d_T)}.
\end{align}
By applying the Hardy–Littlewood theorem only in the time direction and Lemma \ref{lem_TLP} in the spatial direction,
\begin{equation} \label{eq10161717}
    \|\gamma^{-1/p_0}(\bM_t \widetilde{\bT}_{2^{-1}} |f|^{p_0})^{1/p_0}\|_{L_p(\bR^d_T)} \leq N \gamma^{-1/p_0}\|f\|_{L_p(\bR^d_T)}
\end{equation}
and
\begin{align} \label{eq10161718}
   & \sum_{k=0}^\infty 2^{-k\sigma}\left\|\gamma^{-1/p_1} \left(\bM_t \bT_{2^{-k-4}}(|\cL u|^{p_0})(t,x)\right)^{1/p_0} \right\|_{L_p(\bR^d_T)} \nonumber
   \\
   &\leq N\gamma^{-1/p_1} \sum_{k=0}^\infty 2^{-k\sigma}  \| \cL u \|_{L_p(\bR^d_T)} = N\gamma^{-1/p_1} \| \cL u \|_{L_p(\bR^d_T)}.
\end{align}
Combining \eqref{eq10161716}, \eqref{eq10161717}, and \eqref{eq10161718}, we can estimate the last term of \eqref{eq10161724} to obtain
\begin{align*}
    \|\cL u\|_{L_p(\bR^d_T)}^p \leq N\gamma^{1-p/p_0}\|f\|_{L_p(\bR^d_T)}^p + N\gamma^{1-p/p_1}\| \cL u \|_{L_p(\bR^d_T)}^p.
\end{align*}
Since $p<p_1$ and $u\in C_c^\infty([0,T]\times\bR^d)$, we can choose a sufficiently small $\gamma\in(0,1)$ so that $N\gamma^{1-p/p_1}<1/2$, which allows us to absorb the last term to the left. By repeating the above argument with $\cA'(\kappa \rho)$ and $\cB'(\rho)$ instead of $\cA(\kappa \rho)$ and $\cB(\rho)$, \eqref{mainest} is proved for $p\in (p_0,p_1)$.
Since the lower bound of $p_1-p$ depends only on $d$ and $\sigma$, one can repeat the procedure finite times to reach the desired exponent $p$. Actually, we remark that in finite steps, we get a $p_0$, which is greater than $d/\sigma+1$, so that we can choose $p_1=\infty$. This implies that the process can be repeated a uniformly bounded number of times for any $p\in[2,\infty)$. The claim is obtained when $\cL_t$ is independent of $t$.

Next, we consider general $\cL_t$, which is allowed to depend on $t$.  
By the above claim, we have \eqref{mainest} with $-(-\Delta)^{\sigma/2}$ instead of $\cL_t$.
Since for each $t>0$, $\cL_t$ can be viewed as a time-independent operator, by \eqref{oper_bdd},
\begin{equation*}
     \|\cL_t u\|_{L_{p}(\bR^d)} \leq N \|(-\Delta)^{\sigma/2}u\|_{L_{p}(\bR^d)}.
\end{equation*}
We raise both sides of this inequality to the power of $p$ and then integrate with respect to $t$. Then we get the estimate for general $\cL_t$.

Lastly, for $p\in(1,2)$, we utilize a duality argument. As in the above case where $p\geq2$, we only need to prove the case when $\cL_t\equiv \cL$ is independent of $t$. 
Let $L^*_t$ and $\cL^*$ be the operators whose corresponding L\'evy measures are given by $\nu_{t}(-dy)$ and $\widetilde{\nu}_t(-dy)$, respectively. Then one can find that $L^*_t$ and $\cL^*$ are the dual operators of $L_t$ and $\cL$, respectively.
For $p'=p/(p-1)\in(2,\infty)$ and $g\in L_{p'}(\bR^d_T)$, one can find a solution $v\in \bH_{p',0}^\sigma(T)$ to
  \begin{equation*}
    \partial_tv = L^*_{T-t} v -\lambda v + g(T-t,x) \text{ in } \bR_T^d.
  \end{equation*}
  Moreover,
  \begin{equation*}
      \|\partial_t v\|_{L_{p'}(\bR_T^d)} + \|\cL^* v\|_{L_{p'}(\bR_T^d)} + \lambda\|v\|_{L_{p'}(\bR_T^d)} \leq N\|g\|_{L_{p'}(\bR_T^d)},
  \end{equation*}
  Due to the duality property of nonlocal operators, we can apply an integration by parts to get
  \begin{align*}
      &\int_0^T \int_{\bR^d} g(t,x) (\cL u)(t,x) dxdt = \int_0^T \int_{\bR^d} g(T-t,x)(\cL u)(T-t,x) dxdt
      \\
      &= \int_0^T \int_{\bR^d} \left( (\partial_t v)(t,x) - (L^*_{T-t}v)(t,x) +\lambda v(t,x) \right) (\cL u)(T-t,x) dxdt
      \\
      &= \int_0^T \int_{\bR^d} (\cL^* v)(t,x) \left( (\partial_t u)(T-t,x) -(L_{T-t}u)(T-t,x) +\lambda u(T-t,x) \right) dxdt
      \\
      &\leq \|f\|_{L_p(\bR^d_T)}\|\cL^*v\|_{L_p(\bR^d_T)} \leq N\|f\|_{L_p(\bR^d_T)}\|g\|_{L_p(\bR^d_T)}.
  \end{align*}
  For the second inequality above, we used the time-independency of $\cL$ as well as the zero initial conditions $u(0,x)=v(0,x)=0$ to justify
  \begin{equation*}
      \int_0^T \int_{\bR^d} (\partial_t v)(t,x) (\cL u)(T-t,x) dxdt = \int_0^T \int_{\bR^d} (\cL^* v)(t,x) (\partial_tu)(T-t,x) dxdt.
  \end{equation*}
  Since $g\in L_{p'}(\bR^d_T)$ is arbitrary, we see that
  \begin{equation*}
      \|\cL u\|_{L_p(\bR^d_T)} \leq N \|f\|_{L_p(\bR^d_T)}.
  \end{equation*}
  Similarly, by repeating the above argument with $\lambda u$ instead of $\cL u$, $\lambda u$ can also be estimated. Lastly, for $\partial_t u$, one just needs to use \eqref{maineq}. Thus, we arrive at \eqref{mainest} when $\cL$ is independent of $t$.

\textbf{3.} Lastly, we aim to prove the continuity of $L_t$. Note that by \cite[Theorem 1]{MP19}, $-(-\Delta)^{\sigma/2}$ is continuous. Thus, it follows from the result in \textbf{2} that we have \eqref{mainest} with $(L_t,-(-\Delta)^{\sigma/2})$ instead of $(\cL_t,L_t)$. This actually proves the desired result.  
  The theorem is proved.     
\end{proof}

\section{A special case in weighted spaces} \label{sec4}

In this section, we prove Theorem \ref{thm_weight}. Thus, throughout this section, we fix $d=1$ and restrict $\sigma$ to the range $(1,2)$. 

Let $p_0\in(1,\infty)$ and $u\in \bH_{p_0,0}^\sigma(T)$ be a solution to \eqref{weighteq}. Similar to the proof of Lemma \ref{lem10151525}, for any $t_0\in(0,T)$ and $R>0$, we decompose $u=v+w$ so that $v,w\in \bH_{p,0}^\sigma(S,t_0)$ with $S=\min\{0,t_0-R^\sigma\}$ satisfy
\begin{equation} \label{eq1129114}
    \partial_t w = L_tw-\lambda w + f \quad \text{ in } (t_0-R^\sigma,t_0)\times \bR,
\end{equation}
and
\begin{equation} \label{eq11222305}
    \partial_t v = L_tv-\lambda v \quad\text{ in } (t_0-R^\sigma,t_0)\times \bR.
\end{equation}
Here, $f$ is extended to be zero for $t<0$.

In the following lemma, we estimate $v$.

\begin{lemma} \label{lem12051053}
    Let $p_0\in(1,\infty)$, $t_0\in (0,\infty)$, $x_0\in\bR$, $R>0$, $\lambda\geq0$, and $S:=\min\{0,t_0-R^\sigma\}$. Suppose that $L_t$ and $\cL_t$ satisfy Assumptions \ref{levy} and \ref{upper}, respectively, and that $\cL_t\equiv\cL$ is independent of $t$. Let $v\in \bH_{p_0,0}^{\sigma}(S,t_0)$ satisfy \eqref{eq11222305}. Then the followings hold.

    (i) For any $p\in(p_0,\infty)$,
    \begin{align} \label{eq11222351}
        (|\cL v|^p)^{1/p}_{Q_{R/2}(t_0,x_0)} \leq N \sum_{z\in\bZ} (1+|z|)^{-\sigma} (|\cL v|^{p_0})^{1/p_0}_{(t_0-R^\sigma,t_0)\times B_{R/2}(x_0+zR)}
    \end{align}
    and
        \begin{align} \label{eq11222352}
        (|\lambda v|^p)^{1/p}_{Q_{R/2}(t_0,x_0)} \leq N \sum_{z\in\bZ} (1+|z|)^{-\sigma} (|\lambda v|^{p_0})^{1/p_0}_{(t_0-R^\sigma,t_0)\times B_{R/2}(x_0+zR)},
    \end{align}
    where $N=N(p_0,p,\sigma,\Lambda,N_0)$.

    (ii) For any $\vartheta\in(0,\sigma)$ so that
        \begin{align} \label{eq11222353}
        [\cL v]_{C^{\vartheta/\sigma,\vartheta}(Q_{R/2}(t_0,x_0))} \leq N R^{-\vartheta}\sum_{z\in\bZ} (1+|z|)^{-\sigma} (|\cL v|^{p_0})^{1/p_0}_{(t_0-R^\sigma,t_0)\times B_{R/2}(x_0+zR)}
    \end{align}
    and
        \begin{align} \label{eq11222354}
        [\lambda v]_{C^{\vartheta/\sigma,\vartheta}(Q_{R/2}(t_0,x_0))} \leq N R^{-\vartheta}\sum_{z\in\bZ} (1+|z|)^{-\sigma} (|\lambda v|^{p_0})^{1/p_0}_{(t_0-R^\sigma,t_0)\times B_{R/2}(x_0+zR)},
    \end{align}
    where $N=N(p_0,\vartheta,\sigma,\Lambda,N_0)$.
\end{lemma}

\begin{proof}
We only prove \eqref{eq11222351} and \eqref{eq11222353} since the proofs of \eqref{eq11222352} and \eqref{eq11222354} follow in the same way.
 
    $(i)$
    By shifting the coordinates, we only prove the case when $x_0=0$.

    Let us define
    \begin{equation} \label{eq11290038}
        1/p_m:=1/p_0-m\sigma/(d+\sigma), \quad m\in\bN.
    \end{equation} 
    We first consider the case when $p\in(p_0,p_1)$ and $p_0\leq d/\sigma+1$.
    
    Let $\chi\in C^\infty(\bR)$ so that $0\leq\chi\leq1$, $\text{supp}(\chi)\subset (t_0-R^\sigma,t_0+R^\sigma)$, and $\chi=1$ in $(t_0-(R/2)^\sigma,t_0)$. Then $\psi(t,x):=v(t,x)\chi(t) \in \bH_{p_0,0}^\sigma(t_0-R^\sigma,t_0)$ satisfies
    \begin{equation*}
    \partial_t\psi = L\psi -\lambda \psi + v\partial_t\chi \quad \text{ in } (t_0-R^\sigma,t_0)\times\bR^d.
  \end{equation*}
  Similar to \eqref{eq10151503}, by \eqref{eq11072241} with $p$ instead of $l$, and \eqref{eq10120058} with $
  \varepsilon=2^{-2\sigma}$,
  \begin{align} \label{eq11251255}
      &(|\cL v|^{p})_{Q_{R/2}(t_0,0)}^{1/{p}} = (|\cL \psi|^{p})_{Q_{R/2}(t_0,0)}^{1/{p}} \nonumber
      \\
      &\leq NR^\vartheta \|\partial_t(\zeta_0\cL \psi)\|_{L_{p_0}((t_0-R^\sigma,t_0)\times \bR^d)} \nonumber 
      \\
      &\quad + NR^\vartheta  \|(-\Delta)^{\sigma/2}(\zeta_0\cL \psi)\|_{L_{p_0}((t_0-R^\sigma,t_0)\times \bR^d)} \nonumber
      \\
        &\leq N R^\vartheta \|(\cL v)\partial_t\chi\|_{L_{p_0}\left((t_0-R^\sigma,t_0)\times B_{R}\right)} \nonumber
        \\
        &\quad + N R^{-(d+\sigma)/p_0} \|\cL\psi\|_{L_{p_0}\left((t_0-R^\sigma,t_0)\times B_{R}\right)}\nonumber
    \\
    &\quad + NR^{-\sigma/p_0}\sum_{k=0}^\infty 2^{-k\sigma} \left(\int_{t_0-R^\sigma}^{t_0} \left|\cT^{R}_{p_0,2^{-k-4}}(\cL \psi) (t,0) \right|^{p_0} dt\right)^{1/p_0},
  \end{align}
  where $\vartheta:=\sigma-(d+\sigma)/p_0$ and $\zeta_0\in C_c^\infty(B_{R})$ is taken from Lemma \ref{lem10142050}.
  As in \eqref{eq10151506}, by using $|\partial_t\chi|\leq NR^{-\sigma}$ and the triangle inequality,
  \begin{align} \label{eq11251256}
       &R^\vartheta \|(\cL v)\partial_t\chi\|_{L_{p_0}\left((t_0-R^\sigma,t_0)\times B_{R}\right)} + R^{-(d+\sigma)/p_0} \|\cL\psi\|_{L_{p_0}\left((t_0-R^\sigma,t_0)\times B_{R}\right)} \nonumber
       \\
       &\leq N (|\cL v|^{p_0})_{Q_{R}(t_0,0)}^{1/p_0} \nonumber
       \\
       &\leq N (|\cL v|^{p_0})^{1/p_0}_{(t_0-R^\sigma,t_0)\times B_{R/2}(-R)} + N (|\cL v|^{p_0})^{1/p_0}_{(t_0-R^\sigma,t_0)\times B_{R/2}(0)} \nonumber
       \\
       &\quad + N (|\cL v|^{p_0})^{1/p_0}_{(t_0-R^\sigma,t_0)\times B_{R/2}(R)}.
  \end{align}

  Next, we consider the term related to $\cT^{R}_{p_0,2^{-k-4}}(\cL \psi)$. Due to \eqref{eq10181920} and \eqref{equiv}, letting $A_z:=\{zR/2\leq y <(z+1)R/2\}$,
\begin{align*}
    &\cT^{R}_{p_0,2^{-k-4}}(\cL \psi)(t,0)
    \\
    &= N2^{-k\sigma}R^{\sigma-d/p_0} \sum_{z\in\bZ} \int_{B_{2^{-k-4}R}^c \cap A_z} \|\cL\psi(t,\cdot)\|_{L_{p_0}(B_{R}(y))} \nu_t(dy)
    \\
    &\leq N2^{-k\sigma} R^{\sigma-d/p_0} \sum_{z\in\bZ} \int_{B_{2^{-k-4}R}^c \cap A_z} \|\cL\psi(t,\cdot)\|_{L_{p_0}(B_{2R}(\frac{zR}{2}))} \nu_t(dy)
    \\
    &\leq N R^{-d/p_0} \|\cL\psi(t,\cdot)\|_{L_{p_0}(B_{2R}(0))} + N 2^{-k\sigma} |z|^{-\sigma} R^{-d/p_0} \sum_{z\neq 0} \|\cL\psi(t,\cdot)\|_{L_{p_0}(B_{2R}(\frac{zR}{2}))}
    \\
    &\leq N R^{-d/p_0}\sum_{z\in\bZ} (1+|z|)^{-\sigma} \|\cL\psi(t,\cdot)\|_{L_{p_0}(B_{2R}(\frac{zR}{2}))}.
\end{align*}
As in \eqref{eq11251256}, by the triangle inequality,
\begin{equation*}
    \sum_{z\in\bZ} (1+|z|)^{-\sigma} \|\cL\psi(t,\cdot)\|_{L_{p_0}(B_{2R}(\frac{zR}{2}))} \leq N \sum_{z\in\bZ} (1+|z|)^{-\sigma} \|\cL\psi(t,\cdot)\|_{L_{p_0}(B_{R/2}(zR))}.
\end{equation*}
Thus, by the Minkowski inequality,
\begin{align*}
    &R^{-\sigma/p_0} \sum_{k=0}^\infty  2^{-k\sigma} \left(\int_{t_0-R^\sigma}^{t_0} \left|\cT^{R}_{p,2^{-k-4}}(\cL \psi) (t,0) \right|^{p_0} dt\right)^{1/p_0}
    \\
    &\leq N \sum_{k=0}^\infty  2^{-k\sigma} \sum_{z\in\bZ} (1+|z|)^{-\sigma}(|\cL v|^{p_0})^{1/p_0}_{(t_0-R^\sigma,t_0)\times B_{R/2}(zR)}
    \\
    &\leq N \sum_{z\in\bZ} (1+|z|)^{-\sigma} (|\cL v|^{p_0})^{1/p_0}_{(t_0-R^\sigma,t_0)\times B_{R/2}(zR)}.
\end{align*}
This, \eqref{eq11251255}, and \eqref{eq11251256} yield the desired result when $p\in (p_0,p_1)$.

Next, we consider the case when $p\in[p_1,p_2)$ and $p_0\leq p_1\leq d/\sigma+1$. Due to H\"older's inequality, it suffices to deal with $p\in(p_1,p_2)$.
Since \eqref{eq11222351} and \eqref{eq11222352} hold for arbitrary $x_0\in\bR$, we have $\partial_t v \in L_p((t_0-(R/2)^{\sigma},t_0)\times B_R)$ for any $p\in (p_0,p_1)$, and thus the fundamental theorem of calculus implies that we also get $v \in L_p((t_0-(R/2)^{\sigma},t_0)\times B_R)$. Hence, the weaker condition presented in Remark \ref{rem1127} is satisfied. Thus, by repeating the above argument with another cutoff function in time, for $p\in(p_1,p_2)$, there is $\tilde{p}\in(p_0,p_1)$ such that
\begin{equation*}
    1/\tilde{p}=1/p-\sigma/(d+\sigma)
\end{equation*}
and
\begin{align*}
        (|\cL v|^p)^{1/p}_{(t_0-(R/4)^\sigma,t_0)\times B_{R/2}} \leq N \sum_{z\in\bZ} (1+|z|)^{-\sigma} (|\cL v|^{\tilde{p}})^{1/\tilde{p}}_{(t_0-(R/2)^\sigma,t_0)\times B_{R/2}(zR)}.
\end{align*}
Thanks to the result for the range $(p_0,p_1)$ (with a shift of the coordinates),
\begin{align} \label{eq12022217}
        &(|\cL v|^p)^{1/p}_{(t_0-(R/4)^\sigma,t_0)\times B_{R/2}} \nonumber
        \\
        &\leq N \sum_{z\in\bZ} (1+|z|)^{-\sigma} \left( \sum_{y\in\bZ} (1+|y|)^{-\sigma} (|\cL v|^{p_0})^{1/p_0}_{(t_0-R^\sigma,t_0)\times B_{R/2}(zR+yR)} \right) \nonumber
        \\
        &= N \sum_{y,z\in\bZ} (1+|z|)^{-\sigma} \left(1+\left|y-z\right|\right)^{-\sigma} (|\cL v|^{p_0})^{1/p_0}_{(t_0-R^\sigma,t_0)\times B_{R/2}(yR)} \nonumber
        \\
        &\leq N \sum_{y\in\bZ} (1+|y|)^{-\sigma} (|\cL v|^{p_0})^{1/p_0}_{(t_0-R^\sigma,t_0)\times B_{R/2}(yR)}.
\end{align}
Here, for the last inequality, we used $\sigma>1$ and
\begin{align*}
    (1+|z|)^{-\sigma} \left(1+\left|y-z\right|\right)^{-\sigma} &\leq N 1_{A} (1+|z|)^{-\sigma} (1+|y|)^{-\sigma} 
    \\
    &\quad + N 1_{A^c} (1+|y|)^{-\sigma} \left(1+\left|y-z\right|\right)^{-\sigma} 
    \\
    &\leq N(1+|y|)^{-\sigma} \left( (1+|z|)^{-\sigma} + \left(1+\left|y-z\right|\right)^{-\sigma}\right),
\end{align*}
where $A:=\{(y,z)\in \bZ\times\bZ: \frac{|y|}{2} \leq \left|y-z\right|\}$.
By a covering argument in the time direction (see e.g. \cite[Remark 2.15]{XX22}), we obtain the desired result.

For general case when $p\in (p_0,\infty)$, we repeat the above process with $p_m$ finite times until $p_i>d/\sigma+1$. We also remark that by \eqref{eq11290038}, one can find that the number of iteration is bounded by a constant depending only on $d,\sigma$, and $p_0$.

$(ii)$
Let us find $p\in(1,\infty)$ so that
\begin{equation*}
    \vartheta=\sigma-(d+\sigma)/p.
\end{equation*}
By \eqref{eq11290052} and \eqref{eq10120058}, instead of \eqref{eq11251255}, one can obtain that
\begin{align*}
    &R^\vartheta [\cL v]_{C^{\vartheta/\sigma,\vartheta}(Q_{R/2}(t_0,x_0))} 
    \\
    &\leq  N R^\vartheta \|(\cL v)\partial_t\chi\|_{L_{p}\left((t_0-R^\sigma,t_0)\times B_{R}\right)} \nonumber
        \\
        &\quad + N R^{-(d+\sigma)/p} \|\cL\psi\|_{L_{p}\left((t_0-(R/2)^\sigma,t_0)\times B_{R}\right)}\nonumber
    \\
    &\quad + NR^{-\sigma/p}\sum_{k=0}^\infty 2^{-k\sigma} \left(\int_{t_0-R^\sigma}^{t_0} \left|\cT^{R}_{p,2^{-k-4}}(\cL \psi) (t,0) \right|^{p} dt\right)^{1/p},
\end{align*}
where $\psi:=v\chi$ is defined as in $(i)$. Then by repeating the argument presented in $(i)$, one can actually obtain \eqref{eq11222353} with $p$ instead of $p_0$. If $p<p_0$, then one just needs to apply H\"older's inequality, and if $p=p_0$, then the proof is complete since this is the case we aim for. When $p>p_0$, similar to \eqref{eq12022217},
\begin{align*}
    &R^\vartheta [\cL v]_{C^{\vartheta/\sigma,\vartheta}(Q_{R/4}(t_0,x_0))} 
    \\
    &\leq N \sum_{z\in\bZ} (1+|z|)^{-\sigma} \left( \sum_{y\in\bZ} (1+|y|)^{-\sigma} (|\cL v|^{p_0})^{1/p_0}_{(t_0-R^\sigma,t_0)\times B_{R/2}(zR+yR)} \right) 
    \\
    &\leq N \sum_{y\in\bZ} (1+|y|)^{-\sigma} (|\cL v|^{p_0})^{1/p_0}_{(t_0-R^\sigma,t_0)\times B_{R/2}(yR)}.
\end{align*}
The lemma is proved.
\end{proof}

\begin{lemma} \label{lem11292047}
    Let $p_0\in(1,\infty)$, $t_0\in (0,\infty)$, $x_0\in\bR$, $R>0$, and $\lambda\geq0$. Suppose that $L_t$ and $\cL_t$ satisfy Assumptions \ref{levy} and \ref{upper}, respectively. Let $w\in \bH_{p_0,0}^{\sigma}(t_0-R^\sigma,t_0)$ satisfy \eqref{eq1129114}. Then for any $R>0$,
    \begin{align*}
    (|\cL w|^{p_0})^{1/p_0}_{Q_{R/2}(t_0,x_0)} + \lambda (|w|^{p_0})^{1/p_0}_{Q_{R/2}(t_0,x_0)} &\leq N\sum_{j=0}^\infty 2^{(1-\sigma)j}(|f|^{p_0})^{1/p_0}_{(t_0-R^\sigma,t_0)\times B_{2^jR}(x_0)},
\end{align*}
where $N=N(\sigma,\Lambda,p_0,N_0)$.
\end{lemma}

\begin{proof}
As in the proof of Lemma \ref{lem12051053}, we may assume that $x_0=0$.
    By \eqref{eq10120058} with $\varepsilon = 2^{-2\sigma}$,
  \begin{align*}
  &(|\cL w|^{p_0})^{1/p_0}_{Q_{R/2}(t_0,0)} + \lambda (|w|^{p_0})^{1/p_0}_{Q_{R/2}(t_0,x_0)}
  \\
  &\leq N(|f|^{p_0})^{1/p_0}_{Q_{R}(t_0,0)} + N R^{-\sigma} (|w|^{p_0})^{1/p_0}_{Q_{R}(t_0,0)}
    \\
    &\quad + NR^{-\sigma-\sigma/p_0} \sum_{k=0}^\infty 2^{-k\sigma} \left( \int_{t_0-R^\sigma}^{t_0} \left|\cT_{p_0,2^{-k-4}}^R w(t,0) \right|^{p_0} dt \right)^{1/p_0}.
  \end{align*}
  For $C_j:=\{2^{j-1}R\leq |y|<2^jR\}$, due to \eqref{eq10181920} and \eqref{equiv},
  \begin{align*}
    &\cT_{p_0,2^{-k-4}}^R w(t,0)
    \\
    &= N 2^{-k\sigma} R^{\sigma-1/p_0} \sum_{j=0}^\infty \int_{B_{2^{-k-4} R}^c\cap C_j} \|w(t,\cdot)\|_{L_{p_0}(B_R(y))} \nu_t(dy)
    \\
    &\leq N 2^{-k\sigma} R^{\sigma} \sum_{j=0}^\infty 2^{j/p_0} \int_{B_{2^{-k-4} R}^c\cap C_j} \left( \aint_{B_{2^{j+1}R}} |w(t,z)|^{p_0} dz \right)^{1/p_0} \nu_t(dy)
    \\
    &\leq N \left( \aint_{B_{2R}} |w(t,z)|^{p_0} dz \right)^{1/p_0} + N 2^{-k\sigma} \sum_{j=0}^\infty 2^{(1/p_0-\sigma)j} \left( \aint_{B_{2^{j+1}R}} |w(t,z)|^{p_0} dz \right)^{1/p_0}.
\end{align*}
Since $1/p_0<1$, we have
\begin{align} \label{eq11291220}
    &(|\cL w|^{p_0})^{1/p_0}_{Q_{R/2}(t_0,0)} + \lambda (|w|^{p_0})^{1/p_0}_{Q_{R/2}(t_0,x_0)} \nonumber
    \\
    &\leq N(|f|^{p_0})^{1/p_0}_{Q_{R}(t_0,0)} + N  \sum_{j=0}^\infty 2^{(1-\sigma)j}R^{-\sigma} (|w|^{p_0})^{1/p_0}_{(t_0-R^\sigma,t_0)\times B_{2^jR}} \nonumber
    \\
    &=: N(|f|^{p_0})^{1/p_0}_{Q_{R}(t_0,0)} + N \sum_{j=0}^\infty 2^{(1-\sigma)j} A_j.
\end{align}

Now we focus on $A_j$. By \eqref{eq110722228} and \eqref{eq10120058} with $\varepsilon = 2^{-2\sigma}$ and with $2^{j+1}R$ in place of $R$,
\begin{align} \label{eq11291209}
    A_j &\leq N(|\partial_t w|^{p_0})^{1/p_0}_{(t_0-R^\sigma,t_0)\times B_{2^jR}} \nonumber
    \\
    &\leq N(|f|^{p_0})^{1/p_0}_{(t_0-R^\sigma,t_0)\times B_{2^{j+1}R}(x_0)} + N\sum_{l=j+1}^\infty 2^{(1-\sigma)l} A_l,
\end{align}
where $N$ is independent of $j$. Let us multiply \eqref{eq11291209} by $2^{(1-\sigma)j}$ and take the sum over $j\geq j_0$. Then we have
\begin{align*}
    &\sum_{j=j_0}^\infty 2^{(1-\sigma)j} A_j 
    \\
    &\leq N \sum_{j=j_0}^\infty 2^{(1-\sigma)j} (|f|^{p_0})^{1/p_0}_{(t_0-R^\sigma,t_0)\times B_{2^{j+1}R}(x_0)} + N \sum_{j=j_0}^\infty \sum_{l=j+1}^\infty 2^{(1-\sigma)(j+l)} A_l
    \\
    &\leq N \sum_{j=j_0}^\infty 2^{(1-\sigma)j} (|f|^{p_0})^{1/p_0}_{(t_0-R^\sigma,t_0)\times B_{2^{j+1}R}(x_0)} + N\frac{2^{(1-\sigma)j_0}}{1-2^{1-\sigma}} \sum_{l=j_0}^\infty 2^{(1-\sigma)l} A_l.
\end{align*}
As $\sigma>1$, if we choose a sufficiently large $j_0$, we have
\begin{equation} \label{eq11291219}
    \sum_{j=j_0}^\infty 2^{(1-\sigma)j} A_j  \leq N \sum_{j=j_0}^\infty 2^{(1-\sigma)j} (|f|^{p_0})^{1/p_0}_{(t_0-R^\sigma,t_0)\times B_{2^{j+1}R}(x_0)}.
\end{equation}
Combining this and \eqref{eq11291209} with $j=0,1,\cdots,j_0-1$, we arrive at \eqref{eq11291219} with $j_0=0$. It remains to apply \eqref{eq11291220}. The lemma is proved.
\end{proof}

\begin{lemma}
      Let $p_0\in(1,\infty)$, $T\in(0,\infty)$, and $\lambda\geq0$.
      Suppose that $L_t$ and $\cL_t$ satisfy Assumptions \ref{levy} and \ref{upper}, respectively, and that $\cL_t\equiv\cL$ is independent of $t$. Let $u\in \bH_{p_0,0}^{\sigma}(0,T)$ satisfy 
      \begin{equation*}
          \partial_t u = L_tu -\lambda u + f \quad \text{ in } \bR_T.
      \end{equation*}
      Then for any $t_0\in (-\infty,T)$, $x_0\in\bR$, $R>0$, $\vartheta\in(0,\sigma)$, and $\rho\in(0,1/4)$,
      \begin{align} \label{eq11292035}
          &(|\cL u - (\cL u)_{Q_{\rho R}(t_0,x_0)}|)_{Q_{\rho R}(t_0,x_0)} + (|\lambda u - (\lambda u)_{Q_{\rho R}(t_0,x_0)}|)_{Q_{\rho R}(t_0,x_0)} \nonumber
          \\
          &\leq N\rho^\vartheta \sum_{j=0}^\infty 2^{(1-\sigma)j} \left((|\cL u|^{p_0})_{(t_0-R^\sigma,t_0)\times B_{2^jR}(x_0)}^{1/p_0} + (|\lambda u|^{p_0})_{(t_0-R^\sigma,t_0)\times B_{2^jR}(x_0)}^{1/p_0} \right) \nonumber
          \\
          &\quad + N\rho^{-(1+\sigma)/p_0} \sum_{j=0}^\infty 2^{(1-\sigma)j} (|f|^{p_0})_{(t_0-R^\sigma,t_0)\times B_{2^jR}(x_0)}^{1/p_0},
      \end{align}
      where $N=N(p_0,\sigma,\Lambda,N_0,\vartheta)$, and $u$ and $f$ are extended to be zero for $t<0$.
\end{lemma}

\begin{proof}
  Upon shifting the coordinates, it suffices to prove \eqref{eq11292035} when $x_0=0$. By Theorem \ref{mainthm}, there exists $w\in \bH_{p,0}^\sigma(t_0-R^\sigma,t_0)$, which satisfies \eqref{eq1129114}. By extending $w$ to be zero for $t<t_0-R^\sigma$, one can see that $v:=u-w\in \bH_{p_0,0}^{\sigma}(S,t_0)$ with $S=\min\{0,t_0-R^\sigma\}$ satisfies \eqref{eq11222305}.

  For $w$, by H\"older's inequality and Lemma \ref{lem11292047},
  \begin{align} \label{eq11292305}
      &(|\cL w - (\cL w)_{Q_{\rho R}(t_0,0)}|)_{Q_{\rho R}(t_0,0)} + (|\lambda w - (\lambda w)_{Q_{\rho R}(t_0,0)}|)_{Q_{\rho R}(t_0,0)} \nonumber
      \\
      &\leq N\rho^{-(1+\sigma)/p_0} \left( (|\cL w|^{p_0})^{1/p_0}_{Q_{R/2}(t_0,0)} + \lambda (|w|^{p_0})^{1/p_0}_{Q_{R/2}(t_0,0)} \right) \nonumber
      \\
      &\leq N \rho^{-(1+\sigma)/p_0} \sum_{j=0}^\infty 2^{(1-\sigma)j}(|f|^{p_0})^{1/p_0}_{(t_0-R^\sigma,t_0)\times B_{2^jR}}.
  \end{align}

  Next, for $v$, by \eqref{eq11222353} and \eqref{eq11222354},
  \begin{align} \label{eq11292300}
      &(|\cL v - (\cL v)_{Q_{\rho R}(t_0,0)}|)_{Q_{\rho R}(t_0,0)} + (|\lambda v - (\lambda v)_{Q_{\rho R}(t_0,0)}|)_{Q_{\rho R}(t_0,0)} \nonumber
      \\
      &\leq N\rho^\vartheta R^\vartheta \left([\cL v]_{C^{\vartheta/\sigma,\vartheta}(Q_{R/4}(t_0,0))} + [\lambda v]_{C^{\vartheta/\sigma,\vartheta}(Q_{R/4}(t_0,0))} \right) \nonumber
      \\
      &\leq N \rho^\vartheta \sum_{z\in\bZ} (1+|z|)^{-\sigma} \left((|\cL v|^{p_0})^{1/p_0}_{Q_{R/2}(t_0,zR)} + (|\lambda v|^{p_0})^{1/p_0}_{Q_{R/2}(t_0,zR)} \right) \nonumber
      \\
      &\leq N\rho^\vartheta \sum_{z\in\bZ} (1+|z|)^{-\sigma} \left((|\cL u|^{p_0})^{1/p_0}_{Q_{R/2}(t_0,zR)} + (|\lambda u|^{p_0})^{1/p_0}_{Q_{R/2}(t_0,zR)} \right) \nonumber
      \\
      &\quad + N\rho^\vartheta \sum_{z\in\bZ} (1+|z|)^{-\sigma} \left((|\cL w|^{p_0})^{1/p_0}_{Q_{R/2}(t_0,zR)} + (|\lambda w|^{p_0})^{1/p_0}_{Q_{R/2}(t_0,zR)} \right) \nonumber
      \\
      &=: I_1 + I_2.
  \end{align}
  Let us estimate $I_1$. Due to the similarity, we only consider the term related to $\cL u$.  By H\"older's inequality,
  \begin{align*}
      &\sum_{|z|\geq1} (1+|z|)^{-\sigma} (|\cL u|^{p_0})^{1/p_0}_{Q_{R/2}(t_0,zR)} 
      \\
      &\leq N\sum_{j=0}^\infty \sum_{|z|=2^j}^{2^{j+1}} |z|^{-\sigma} (|\cL u|^{p_0})^{1/p_0}_{Q_{R/2}(t_0,zR)}
      \\
      &\leq N\sum_{j=0}^\infty 2^{-\sigma j} \sum_{|z|=2^j}^{2^{j+1}} (|\cL u|^{p_0})^{1/p_0}_{Q_{R/2}(t_0,zR)}
      \\
      &\leq N \sum_{j=0}^\infty 2^{(-\sigma+1-1/p_0) j} \left( \sum_{|z|=2^j}^{2^{j+1}} (|\cL u|^{p_0})_{Q_{R/2}(t_0,zR)} \right)^{1/p_0}
      \\
      &\leq N\sum_{j=0}^\infty 2^{(1-\sigma) j} (|\cL u|^{p_0})_{(t_0-R^\sigma,t_0)\times B_{2^jR}}^{1/p_0}.
  \end{align*}
Thus,
  \begin{align} \label{eq11292301}
      I_1 &\leq N\rho^\vartheta \sum_{j=0}^\infty 2^{(1-\sigma)j} \left((|\cL u|^{p_0})_{(t_0-R^\sigma,t_0)\times B_{2^jR}}^{1/p_0} + (|\lambda u|^{p_0})_{(t_0-R^\sigma,t_0)\times B_{2^jR}}^{1/p_0} \right).
  \end{align}
  Now we handle $I_2$. As above, we only estimate $\cL w$. By Lemma \ref{lem11292047},
  \begin{align*}
      &\sum_{|z|\geq1} (1+|z|)^{-\sigma} (|\cL w|^{p_0})^{1/p_0}_{Q_{R/2}(t_0,zR)} 
      \\
      &\leq N\sum_{|z|\geq1} \sum_{j=0}^\infty 2^{(1-\sigma)j} (1+|z|)^{-\sigma}  (|f|^{p_0})^{1/p_0}_{(t_0-R^\sigma,t_0)\times B_{2^jR}(zR)}
      \\
      &\leq N\sum_{k=0}^\infty \sum_{|z|=2^k}^{2^{k+1}} \sum_{j=0}^\infty 2^{(1-\sigma)j} (1+|z|)^{-\sigma}  (|f|^{p_0})^{1/p_0}_{(t_0-R^\sigma,t_0)\times B_{2^jR}(zR)}.
  \end{align*}
  For $k\geq0$, by H\"older's inequality,
  \begin{align*}
      &\sum_{|z|=2^k}^{2^{k+1}} (1+|z|)^{-\sigma}  (|f|^{p_0})^{1/p_0}_{(t_0-R^\sigma,t_0)\times B_{2^jR}(zR)}
      \\
      &\leq N2^{-k\sigma} \sum_{|z|=2^k}^{2^{k+1}}(|f|^{p_0})^{1/p_0}_{(t_0-R^\sigma,t_0)\times B_{2^jR}(zR)}
      \\
      &\leq N 2^{(-\sigma+1-1/p_0)k}\left(\sum_{|z|=2^k}^{2^{k+1}} (|f|^{p_0})_{(t_0-R^\sigma,t_0)\times B_{2^jR}(zR)} \right)^{1/p_0}
      \\
      &\leq N 2^{(-\sigma+1)k} \Big( 1_{k\geq j} (|f|^{p_0})^{1/p_0}_{(t_0-R^\sigma,t_0)\times B_{2^{k+2}R}} + 1_{k<j} (|f|^{p_0})^{1/p_0}_{(t_0-R^\sigma,t_0)\times B_{2^{j+2}R}} \Big).
  \end{align*}
  Thus,
  \begin{align*}
      &\sum_{k=0}^\infty \sum_{|z|=2^k}^{2^{k+1}} \sum_{j=0}^\infty 2^{(1-\sigma)j} (1+|z|)^{-\sigma}  (|f|^{p_0})^{1/p_0}_{(t_0-R^\sigma,t_0)\times B_{2^jR}(zR)}
      \\
      &\leq N\sum_{k=0}^\infty 2^{(1-\sigma)k} \sum_{j=0}^\infty 2^{(1-\sigma)j} \Big( 1_{k\geq j} (|f|^{p_0})^{1/p_0}_{(t_0-R^\sigma,t_0)\times B_{2^{k+2}R}} 
      \\
      &\qquad \qquad \qquad \qquad \qquad \qquad \quad+ 1_{k<j} (|f|^{p_0})^{1/p_0}_{(t_0-R^\sigma,t_0)\times B_{2^{j+2}R}} \Big)
      \\
      &\leq N \sum_{k=0}^\infty 2^{(1-\sigma)k} (|f|^{p_0})^{1/p_0}_{(t_0-R^\sigma,t_0)\times B_{2^{k}R}},
  \end{align*}
  which yields
  \begin{equation} \label{eq11292302}
      I_2 \leq N \rho^\vartheta\sum_{k=0}^\infty 2^{(1-\sigma)k} (|f|^{p_0})^{1/p_0}_{(t_0-R^\sigma,t_0)\times B_{2^{k}R}}.
  \end{equation}
  Hence, by \eqref{eq11292300}, \eqref{eq11292301}, and \eqref{eq11292302},
  \begin{align} \label{eq11292306}
      &(|\cL v - (\cL v)_{Q_{\rho R}(t_0,0)}|)_{Q_{\rho R}(t_0,0)} + (|\lambda v - (\lambda v)_{Q_{\rho R}(t_0,0)}|)_{Q_{\rho R}(t_0,0)} \nonumber 
      \\
      &\leq N\rho^\vartheta \sum_{j=0}^\infty 2^{(1-\sigma)j} \left((|\cL u|^{p_0})_{(t_0-R^\sigma,t_0)\times B_{2^jR}}^{1/p_0} + (|\lambda u|^{p_0})_{(t_0-R^\sigma,t_0)\times B_{2^jR}}^{1/p_0} \right) \nonumber
      \\
      &\quad + N\rho^\vartheta \sum_{k=0}^\infty 2^{(1-\sigma)k} (|f|^{p_0})^{1/p_0}_{(t_0-R^\sigma,t_0)\times B_{2^{k}R}}.
  \end{align}
  Since $u=v+w$, \eqref{eq11292305} and \eqref{eq11292306} leads to \eqref{eq11292035}. The lemma is proved.  
\end{proof}

Before we present the proof of Theorem \ref{thm_weight}, we introduce the following dyadic cubes. For each $n \in \bZ$, we assign an integer $k(n)\in\bZ$ such that
\begin{equation*}
    k(n) \le \sigma n < k(n)+1.
\end{equation*}
Let
\begin{equation*}
    Q^{n}_{\mathbf{i}}
=
\bigg[ \frac{i_0}{2^{k(n)}} + T, \frac{i_0+1}{2^{k(n)}} + T \bigg)
\times
\bigg[ \frac{i_1}{2^{n}}, \frac{i_1+1}{2^{n}} \bigg)
\times \cdots \times
\bigg[ \frac{i_d}{2^{n}}, \frac{i_d+1}{2^{n}} \bigg),
\end{equation*}
where $\mathbf{i}=(i_0,\dots,i_d)\in \bZ^{d+1}$ and $i_0\leq -1$.
Next, the dyadic sharp function of $u$ is defined by
\begin{equation} \label{eq1130123}
    u^{\sharp}_{dy}(t,x)
=
\sup_{n<\infty}
\int_{Q^{n}_{\mathbf{i}} \ni (t,x)}
|u(s,y) - u_{|n}(t,x)|  dyds,
\end{equation}
where
\begin{equation*}
    u_{|n}(t,x)
=
\int_{Q^{\,n}_{\mathbf{i}}} u(s,y) dyds
\quad \text{ for } (t,x) \in Q^{n}_{\mathbf{i}}.
\end{equation*}

\begin{proof}[Proof of Theorem \ref{thm_weight}]

We first note that the continuity of $L_t=-(-\Delta)^{\sigma/2}$ from $H_{p,\omega_2}^{\sigma}(\bR)$ to $L_{p,\omega_2}(\bR)$ is presented in \cite[Theorem 2.5]{DJK23}. Thus, we only need to prove a version of \textbf{1} in the proof of Theorem \ref{mainthm}: we assume that $L_t$ is continuous and aim to obtain \eqref{weightest} for $u\in C_c^\infty([0,T]\times \bR)$ with $u(0,x)=0$.

Let $\omega(t,x)=\omega_1(t)\omega_2(x)$ such that $\omega_1\in A_q(\bR)$ and $\omega_2\in A_p(\bR)$. By reverse H\"older's inequality for Muckenhoupt weights (see e.g. \cite[Corollary 7.2.6]{G14}), there are $\gamma_1=\gamma_1(q,K_0)$ and $\gamma_2=\gamma_2(p,K_0)$ so that $q-\gamma_1>1$, $p-\gamma_2>1$, and
\begin{equation*}
    \omega_1 \in A_{q-\gamma_1}(\bR), \quad \omega_2 \in A_{p-\gamma_2}(\bR).
\end{equation*}
Due to the relation $A_{r_1}(\bR)\subset A_{r_2}(\bR)$ for $r_1\leq r_2$ (see e.g. \cite[Proposition 7.1.5]{G14}), letting
\begin{equation*}
    p_0:=\min\left\{ \frac{q}{q-\gamma_1}, \frac{p}{p-\gamma_2} \right\} \in (1,\infty),
\end{equation*}
we have
\begin{equation*}
    \omega_1 \in A_{q-\gamma_1}(\bR) \subset A_{q/p_0}(\bR), \quad \omega_2 \in A_{p-\gamma_2}(\bR) \subset A_{p/p_0}(\bR).
\end{equation*}

Let $(t_0,x_0) \in (-\infty,T)\times \bR$. Then for any $Q^{n}_{\mathbf{i}}$ containing $(t_0,x_0)$, one can find $R=R(\sigma,n)>0$ such that
\begin{equation*}
    Q^{n}_{\mathbf{i}} \subset Q_R(t_1,x_0), \quad |Q_R(t_1,x_0)| \leq N(\sigma) |Q^{n}_{\mathbf{i}}|,
\end{equation*}
where $t_1 := \min\{T,t_0+R^\sigma/2\}$.
Since
\begin{equation*}
    \left|\aint_A h \, dxdt - \aint_B h \,dxdt \right| \leq \frac{|B|}{|A|}\aint_{B} |h-(h)_B| \, dxdt, \quad A\subset B,
\end{equation*}
by applying \eqref{eq11292035} with $t_1$ in place of $t_0$, and using $\sigma\in(1,2)$,
\begin{align*}
    &\aint_{Q^{n}_{\mathbf{i}}\ni (t_0,x_0)} \left(|\cL u(t,x)- (\cL u)_{|n}(t_0,x_0)| + |\lambda u(t,x)- (\lambda u)_{|n}(t_0,x_0)| \right) dxdt
    \\
    &\leq N\rho^\vartheta \sum_{j=0}^\infty 2^{(1-\sigma)j} \left((|\cL u|^{p_0})_{(t_0-R^\sigma,t_0)\times B_{2^j\rho^{-1}R}(x_0)}^{1/p_0} + (|\lambda u|^{p_0})_{(t_0-R^\sigma,t_0)\times B_{2^j\rho^{-1}R}(x_0)}^{1/p_0} \right) \nonumber
    \\
     &\quad + N\rho^{-(1+\sigma)/p_0} \sum_{j=0}^\infty 2^{(1-\sigma)j} (|f|^{p_0})_{(t_0-R^\sigma,t_0)\times B_{2^j\rho^{-1}R}(x_0)}^{1/p_0}
    \\
     &\leq N \rho^\vartheta (\cS\cM|\cL u|^{p_0})^{1/p_0}(t_0,x_0) + N \rho^\vartheta (\cS\cM|\lambda u|^{p_0})^{1/p_0}(t_0,x_0) 
     \\
     &\quad + N\rho^{-(1+\sigma)/p_0}(\cS\cM|f|^{p_0})^{1/p_0}(t_0,x_0),
\end{align*}
where $\vartheta\in(0,\sigma)$, and $\cS\cM h$ is the strong maximal function defined by
\begin{equation*}
    (\cS\cM h)(t_0,x_0) := \sup_{(t-R_1^\sigma,t)\times B_{R_2}(x) \ni (t_0,x_0)} \aint_{(t-R_1^\sigma,t)\times B_{R_2}(x)} |h(r,z)|1_{r<T} drdz.
\end{equation*}
Thus, by \eqref{eq1130123},
\begin{align*}
    &(\cL u)^{\sharp}_{dy}(t_0,x_0) + (\lambda u)^{\sharp}_{dy}(t_0,x_0)
    \\
    &\leq N \rho^\vartheta (\cS\cM|\cL u|^{p_0})^{1/p_0}(t_0,x_0) + N \rho^\vartheta (\cS\cM|\lambda u|^{p_0})^{1/p_0}(t_0,x_0) 
    \\
    &\quad + N\rho^{-(1+\sigma)/p_0}(\cS\cM|f|^{p_0})^{1/p_0}(t_0,x_0).
\end{align*}
Now we apply the weighted sharp function theorem and the weighted maximal function theorem (see \cite[Corollary 2.7]{DK18} and \cite[Theorem 5.2]{DK21}) to get
\begin{align*}
    \|\cL u\|_{L_{q,p,\omega}(\bR_T)} + \lambda\|u\|_{L_{q,p,\omega}(\bR_T)} &\leq N\rho^\vartheta (\|\cL u\|_{L_{q,p,\omega}(\bR_T)} + \lambda\|u\|_{L_{q,p,\omega}(\bR_T)}) 
    \\
    &\quad + N\rho^{-(1+\sigma)/p_0} \|f\|_{L_{q,p,\omega}(\bR_T)}.
\end{align*}
Since $N$ is independent of $\rho$, one can take a sufficiently small $\rho\in(0,1/4)$ to obtain
\begin{equation*}
    \|\cL u\|_{L_{q,p,\omega}(\bR_T)} + \lambda\|u\|_{L_{q,p,\omega}(\bR_T)} \leq N\|f\|_{L_{q,p,\omega}(\bR_T)}.
\end{equation*}
Here, we remark that one may choose $\vartheta=\sigma/2\in(0,\sigma)$ to remove the dependence of $N$ on $\vartheta$.
Finally, the estimation of $\partial_t u$ follows directly from \eqref{weighteq}.
    The theorem is proved.
\end{proof}

\appendix 

\section{Miscellaneous lemmas} \label{secA}

\begin{lem}
  Let $\nu_t$ be a family of L\'evy measures satisfying Assumption \ref{upper}.

  $(i)$ For any $c<\sigma$, there exists $N=N(\sigma,\Lambda)$ such that
  \begin{equation} \label{eq8270021}
  \int_{B_r^c} |y|^c \nu_t(dy) \leq N r^{c-\sigma}.
\end{equation}

$(ii)$ For any $c>\sigma$, there exists $N=N(c,\Lambda)$ such that
\begin{equation} \label{eq9251355}
  \int_{B_{r}} |y|^c \nu_t(dy) \leq N r^{c-\sigma}.
\end{equation}
\end{lem}

\begin{proof}
  $(i)$ Due to \eqref{equiv}, we have
  \begin{align*}
    \int_{B_r^c} |y|^c \nu_t(dy) &= \sum_{k=0}^\infty \int_{B_{2^{k+1}r}\setminus B_{2^kr}} |y|^c \nu_t(dy) \leq N r^c\sum_{k=0}^\infty 2^{kc} \int_{B_{2^kr}^c} \nu_t(dy) 
    \\
    &\leq N r^{c-\sigma}\sum_{k=0}^\infty 2^{k(c-\sigma)}  \leq N(\sigma,\Lambda) r^{c-\sigma}.
  \end{align*}

  $(ii)$ As in $(i)$, by \eqref{equiv},
  \begin{align*}
     \int_{B_{r}} |y|^c \nu_t(dy) &\leq r^{c} \sum_{k=0}^\infty 2^{-kc} \int_{B_{2^{-k}r}\setminus B_{2^{-k-1}r}} \nu_t(dy) \leq N r^{c-\sigma} \sum_{k=0}^\infty 2^{k(\sigma-c)} = N r^{c-\sigma}.
  \end{align*}
\end{proof}

\begin{lem}
  Let
    \begin{equation*}
    m(t,\xi):=\int_{\bR^d} \left(e^{i\xi\cdot y}-1 - i\xi\cdot y^{(\sigma)}\right) \nu_t(dy).
  \end{equation*}

  $(i)$ If $\nu_t$ satisfies Assumption \ref{upper}, then there exists $N=N(\sigma,\Lambda)$ such that
  \begin{equation} \label{ineq9251512}
    |m(t,\xi)|\leq N|\xi|^{\sigma}.
  \end{equation}

  $(ii)$ If $\nu_t$ satisfies Assumption \ref{levy}, then there exists $N=N(N_0)$ such that
  \begin{equation} \label{ineq9260032}
    -\mathrm{Re}(m(t,\xi))\geq N|\xi|^\sigma.
  \end{equation}
\end{lem}

\begin{proof}
$(i)$  Note that
  \begin{align*}
    |m(t,\xi)| &=\int_{\bR^d} \left( 1-\cos(\xi\cdot y) \right) \nu_t(dy) 
    \\
    &\quad + \left|\int_{\bR^d} (\sin(\xi\cdot y) - \xi\cdot y^{(\sigma)}) \nu_t(dy) \right| =:I_1+I_2.
  \end{align*}
  We first estimate $I_1$. Using $1-\cos(x)\leq \min\{1,|x|^2\}$, \eqref{eq8270021}, and \eqref{eq9251355}, we have
  \begin{align*}
    I_1\leq \int_{|y|\leq|\xi|^{-1}} |\xi|^2|y|^2 \nu_t(dy) + \int_{|y|>|\xi|^{-1}} \nu_t(dy) \leq N |\xi|^{\sigma}.
  \end{align*}
   Similarly, by using $|x-\sin(x)|\leq N\min\{|x|,|x|^3\}$ for $\sigma>1$, and $|\sin(x)|\leq \min\{1,|x|\}$ for $\sigma<1$, one can easily show
  \begin{equation*}
    I_2 \leq N|\xi|^\sigma.
  \end{equation*}
  For $\sigma=1$, due to \eqref{upper_mu}, \eqref{cancel_mu}, and \eqref{eq9251355},
  \begin{align*}
      I_2 &\leq \int_{|y|\leq |\xi|^{-1}} |\sin(\xi\cdot y) - \xi\cdot y^{(\sigma)}| \nu_t(dy) + \int_{|y| >|\xi|^{-1}} |\sin(\xi\cdot y)| \nu_t(dy)
      \\
      &\leq N\left( |\xi|^3 \int_{|y|\leq |\xi|^{-1}} |y|^3 \nu_t(dy) + |\xi| \right) \leq N|\xi|.
  \end{align*}
  Thus, \eqref{ineq9251512} is proved.

$(ii)$ Since $1-\cos(x)\geq |x|^2/3$ for $|x|\leq 1$, we have
\begin{align*}
  -\text{Re}(m(t,\xi)) &= \int_{\bR^d} \left( 1-\cos(\xi\cdot y) \right) \nu_t(dy) \geq \frac{1}{3}\int_{|\xi\cdot y|\leq 1} |\xi\cdot y|^2 \nu_t(dy).
\end{align*}
Thus, by \eqref{nonde}, we have the desired result. The lemma is proved.
\end{proof}

  Recall that
  \begin{equation*}
     (\cT_{p,\kappa}^R u)(t,x) = \kappa^\sigma R^{\sigma-d/p} \int_{B_{\kappa R}^c} \|u(t,\cdot)\|_{L_p(B_R(x+y))} \nu_t(dy). 
\end{equation*}

\begin{lemma}
  Let $p\in(1,\infty)$, $R>0$, $k\in\bN_0$,
  \begin{align*}
    r_k:=R(1-2^{-k-1}),
  \end{align*}
  and $\zeta_k \in C_c^\infty(B_{r_{k+1}})$ such that $0\leq\zeta_k\leq1$, $\zeta_k=1$ in $B_{r_k}$, and
  \begin{align*}
    |D_x\zeta_k|\leq N\frac{2^k}{R}, \quad |D_x^2\zeta_k|\leq N \frac{2^{2k}}{R^2}.
  \end{align*}
  Then, under Assumption \ref{upper}, the following estimates hold for $I_k:=\|L_t(\zeta_k u)-\zeta_kL_tu\|_{L_p(\bR_T^d)}$.

  (i) If $\sigma\in(0,1)$, then 
  \begin{align} \label{eq10102353}
    I_k\leq N \frac{2^{k\sigma}}{R^\sigma} \|u\|_{L_p\left((0,T)\times B_{R}\right)} + N\frac{2^{k\sigma}}{R^{\sigma-d/p}} \left( \int_0^T \left|\cT_{p,2^{-k-4}}^R u(t,0) \right|^p dt \right)^{1/p}.
  \end{align}

  (ii) If $\sigma\in(1,2)$, then
  \begin{align} \label{eq10121035}
      I_k &\leq N\frac{2^{k(\sigma-1)}}{R^{\sigma-1}} \|\nabla u\|_{L_p\left((0,T)\times B_{r_{k+3}}\right)} + N \frac{2^{k\sigma}}{R^\sigma} \|u\|_{L_p\left((0,T)\times B_{R}\right)} \nonumber
      \\
      &\quad + N\frac{2^{k\sigma}}{R^{\sigma-d/p}} \left( \int_0^T \left|\cT_{p,2^{-k-4}}^R u(t,0) \right|^p dt \right)^{1/p}.
  \end{align}

  (iii) If $\sigma=1$, then
  for any $\varepsilon\in(0,1]$,
\begin{align} \label{eq10131542}
    I_k &\leq \varepsilon^3\|\nabla u\|_{L_p\left((0,T)\times B_{r_{k+3}}\right)} + N\varepsilon^{-3}\frac{2^k}{R} \|u\|_{L_p\left((0,T)\times B_{R}\right)} \nonumber
    \\
    &\quad + N\frac{2^{k\sigma}}{R^{\sigma-d/p}}\left( \int_0^T \left|\cT_{p,2^{-k-4}}^R u(t,0) \right|^p dt \right)^{1/p}.
\end{align}
  Here, in all the three cases, $N$ depends only on $\sigma, d,p$, and $\Lambda$.
\end{lemma}

\begin{proof}
First, note that
\begin{align} \label{eq8270009}
  &L_t(\zeta_k u)-\zeta_kL_tu \nonumber
  \\
  &= \int_{\bR^d} \bigg( (\zeta_k(x+y)-\zeta_k(x))u(t,x+y)  - y^{(\sigma)}\cdot \nabla\zeta_k(x) u(t,x) \bigg) \nu_t(dy).
\end{align}
We estimate this integral according to the value of $\sigma$.

$(i)$ Let $\sigma\in(0,1)$ and $\tilde{r}_k=r_{k+3}-r_{k+2}=2^{-k-4}R$. By \eqref{eq8270009},
\begin{align*}
  |L_t(\zeta_k u)-\zeta_kL_tu| &\leq \int_{\bR^d} |\zeta_k(x+y)-\zeta_k(x)||u(t,x+y)| \nu_t(dy)
  \\
  &=\int_{B_{\tilde{r}_k}} +\int_{B_{\tilde{r}_k}^c} =: J_{k1}^1+J_{k2}^1.
\end{align*}
We first estimate $J_{k1}^1$. 
Note that $J_{k1}^1=0$ when $|x|>r_{k+2}$, and for $x\in B_{r_{k+2}}$ and $y\in B_{\tilde{r}_k}$,
\begin{equation} \label{eq10022329}
  |\zeta_k(x+y)-\zeta_k(x)|\leq N \|D_x \zeta_k\|_{L_\infty(\bR^d)}|y|1_{|x|\leq r_{k+2}} \leq N\frac{2^k}{R}|y|1_{|x|\leq r_{k+2}}.
\end{equation}
Thus, by the Minkowski inequality and \eqref{eq9251355},
\begin{align} \label{eq10030026}
  \|J_{k1}^1(t,\cdot)\|_{L_p(\bR^d)} &\leq N\frac{2^k}{R} \int_{B_{\tilde{r}_k}} \|u(t,\cdot+y)\|_{L_p\left(B_{r_{k+2}}\right)} |y| \nu_t(dy) \nonumber
  \\
  &\leq N\frac{2^k}{R} \|u(t,\cdot)\|_{L_p\left(B_{r_{k+3}}\right)} \int_{B_{\tilde{r}_k}} |y| \nu_t(dy) \nonumber
  \\
  &\leq N\frac{2^k}{R} \|u(t,\cdot)\|_{L_p\left(B_{r_{k+3}}\right)} \tilde{r}_k^{1-\sigma} = N\frac{2^{k\sigma}}{R^\sigma} \|u(t,\cdot)\|_{L_p\left(B_{R}\right)}.
\end{align}

To conclude \eqref{eq10102353}, it remains to estimate $J_{k2}^1$, which can be decomposed as
\begin{align*}
  J_{k2}^1 \leq \int_{B_{\tilde{r}_k}^c}  \left(1_{|x+y|<r_{k+1}}+1_{|x|<r_{k+1}}\right)|u(t,x+y)| \nu_t(dy) =: J_{k21}^1 + J_{k22}^1.
\end{align*}
For $J_{k21}^1$, by the Minkowski inequality and \eqref{eq8270021},
\begin{align} \label{eq10101149}
  \|J_{k2}^1(t,\cdot)\|_{L_p(\bR^d)} &\leq N \|u(t,\cdot)\|_{L_p(B_{r_{k+1}})} \int_{B_{\tilde{r}_k}^c} \nu_t(dy) \nonumber
  \\
  &\leq N \tilde{r}_k^{-\sigma}\|u(t,\cdot)\|_{L_p(B_{r_{k+1}})} \leq N \frac{2^{k\sigma}}{R^\sigma} \|u(t,\cdot)\|_{L_p(B_{R})}.
\end{align}
Lastly, by the Minkowski inequality and the inequality $r_{k+1}\leq R$,
\begin{align} \label{eq10101529}
    \|J_{k22}^1(t,\cdot)\|_{L_p(\bR^d)} &\leq \int_{B_{\tilde{r}_k}^c} \|u(t,\cdot)\|_{L_p(B_{r_{k+1}}(y))} \nu_t(dy) \nonumber
    \\
    &\leq \int_{B_{\tilde{r}_k}^c} \|u(t,\cdot)\|_{L_p(B_{R}(y))} \nu_t(dy) \nonumber
    \\
    &= \frac{2^{(k+4)\sigma}}{R^{\sigma-d/p}} \cT_{p,2^{-k-4}}^R u(t,0).
\end{align}
This together with \eqref{eq10030026} and \eqref{eq10101149} completes the estimation of $I_k$.

$(ii)$ Now we consider the case $\sigma\in(1,2)$. By \eqref{eq8270009},
\begin{align*}
  |L_t(\zeta_k u)-\zeta_kL_tu|   &\leq \int_{B_{\tilde{r}_k}} |\zeta_k(x+y)-\zeta_k(x)||u(t,x+y)-u(t,x)| \nu_t(dy)
  \\
  &\quad+ \int_{B_{\tilde{r}_k}} |\zeta_k(x+y)-\zeta_k(x) - y \cdot \nabla\zeta_k(x)||u(t,x)| \nu_t(dy)
  \\
  &\quad + \int_{B_{\tilde{r}_k}^c} |\zeta_k(x+y)-\zeta_k(x)||u(t,x+y)| \nu_t(dy)
  \\
  &\quad + \int_{B_{\tilde{r}_k}^c} |y\cdot \nabla\zeta_k(x)||u(t,x)| \nu_t(dy)
  \\
  &=:J^2_{k1}+J^2_{k2}+ J^2_{k3} + J^2_{k4}.
\end{align*}
For $J^2_{k1}$, by \eqref{eq10022329} and the fundamental theorem of calculus,
\begin{align*}
  J^2_{k1} \leq N \frac{2^k}{R} 1_{|x|<r_{k+2}} \int_{B_{\tilde{r}_k}} \int_0^1 |\nabla u(t,x+sy)| |y|^2 dsdy.
\end{align*}
Thus, a similar computation as in \eqref{eq10030026} leads to
\begin{align} \label{eq10110046}
  \|J_{k1}^2\|_{L_p(\bR_T^d)} \leq N\frac{2^{k(\sigma-1)}}{R^{\sigma-1}} \|\nabla u\|_{L_p\left((0,T)\times B_{r_{k+3}}\right)}.
\end{align}

Next, we deal with $J^2_{k2}$. Note that if $|x|\geq r_{k+2}$ and $y\in B_{\tilde{r}_k}$, then
\begin{equation*}
    |x+y| > |x|-|y| > 2r_{k+2} -r_{k+3} > r_{k+1},
\end{equation*}
which leads to $\zeta_k(x+y)-\zeta_k(x)-y \cdot \nabla\zeta_k(x)=0$. Also,
\begin{align*}
  |\zeta_k(x+y)-\zeta_k(x)-y \cdot \nabla\zeta_k(x)| \leq \|D^2_x\zeta_k\|_{L_\infty}|y|^2 \leq N \frac{2^{2k}}{R^2}|y|^2.
\end{align*}
Hence, by \eqref{eq9251355},
\begin{align*}
  J^2_{k2} &\leq N\frac{2^{2k}}{R^2}|u(t,x)|1_{|x|<r_{k+2}} \int_{B_{\tilde{r}_k}} |y|^2 \nu_t(dy) \leq N\frac{2^{k\sigma}}{R^\sigma} |u(t,x)|1_{|x|<r_{k+2}},
\end{align*}
which easily yields
\begin{align} \label{eq10111131}
  \|J_{k2}^2\|_{L_p(\bR_T^d)} \leq N\frac{2^{k\sigma}}{R^\sigma} \|u\|_{L_p\left((0,T)\times B_{R}\right)}.
\end{align}

For $J^2_{k3}$, as in \eqref{eq10101149} and \eqref{eq10101529},
\begin{align} \label{eq10192142}
  \|J_{k3}^2\|_{L_p(\bR_T^d)} &\leq N\frac{2^{k\sigma}}{R^\sigma} \|u\|_{L_p\left((0,T)\times B_{R}\right)} + N\frac{2^{k\sigma}}{R^{\sigma-d/p}} \left( \int_0^T \left|\cT_{p,2^{-k-4}}^R u(t,0) \right|^p dt \right)^{1/p}.
\end{align}

Lastly, by \eqref{eq8270021},
\begin{align*}
    J_{k4}^2 \leq N\frac{2^k}{R}|u(t,x)|1_{|x|<r_{k+1}} \int_{B_{\tilde{r}_k}^c} |y| \nu_t(dy) \leq N\frac{2^{k\sigma}}{R^\sigma}|u(t,x)|1_{|x|<r_{k+1}}.
\end{align*}
Thus,
\begin{align} \label{eq10192141}
  \|J_{k4}^2\|_{L_p(\bR_T^d)} \leq N\frac{2^{k\sigma}}{R^\sigma} \|u\|_{L_p\left((0,T)\times B_{R}\right)}.
\end{align}
Now \eqref{eq10121035} follows from \eqref{eq10110046}, \eqref{eq10111131}, \eqref{eq10192142}, and \eqref{eq10192141}.

$(iii)$ Lastly, we assume that $\sigma=1$. Let $c\in(0,1)$ be a constant which will be determined below. By using the cancellation condition \eqref{cancel_mu}, for $\delta_k:=c\tilde{r}_{k}$,
\begin{align*}
  &|L_t(\zeta_k u)-\zeta_kL_tu| 
  \\&\leq \int_{B_{\delta_k}} |\zeta_k(x+y)-\zeta_k(x)||u(x+y)-u(x)| \nu_t(dy)
  \\
  &\quad+ \int_{B_{\delta_k}} |\zeta_k(x+y) - \zeta_k(x) - y \cdot \nabla\zeta_k(x)||u(t,x)| \nu_t(dy)
  \\
  &\quad+\int_{\delta_k\leq |y| < \tilde{r}_k} |\zeta_k(x+y)-\zeta_k(x)||u(x+y)| \nu_t(dy)
  \\
  &\quad+\int_{B_{\tilde{r}_k}^c} |\zeta_k(x+y)-\zeta_k(x)||u(x+y)| \nu_t(dy)
  \\
  &=:J_{k1}^3 + J_{k2}^3 + J_{k3}^3 + J_{k4}^3
\end{align*}
Similar to \eqref{eq10110046}, \eqref{eq10111131}, \eqref{eq10101149}, and \eqref{eq10101529},
\begin{align*}
    \|J_{k1}^3\|_{L_p(\bR^d_T)} \leq Nc\|\nabla u\|_{L_p\left((0,T)\times B_{r_{k+3}}\right)}, \quad \|J_{k2}^3\|_{L_p(\bR^d_T)} \leq N c\frac{2^k}{R} \|u\|_{L_p((0,T)\times B_{R})},
\end{align*}
and
\begin{align*}
    \|J_{k4}^3\|_{L_p(\bR^d_T)} \leq N \frac{2^{k}}{R} \|u\|_{L_p((0,T)\times B_{R})} + N \frac{2^{k}}{R^{1-d/p}} \left( \int_0^T \left|\cT_{p,2^{-k-4}}^R u(t,0) \right|^p dt \right)^{1/p}.
\end{align*}
For $J_{k3}^3$, by \eqref{eq10022329},
\begin{align*}
      \|J_{k3}^3(t,\cdot)\|_{L_p(\bR^d)} &\leq N\frac{2^k}{R} \int_{\delta_k\leq |y| < \tilde{r}_k} \|u(t,\cdot+y)\|_{L_p\left(B_{r_{k+2}}\right)} |y| \nu_t(dy)
  \\
  &\leq N\frac{2^k}{R} \|u(t,\cdot)\|_{L_p\left(B_{r_{k+3}}\right)} \int_{\delta_k\leq |y| < \tilde{r}_k} |y| \nu_t(dy) 
    \\
  &\leq N\|u(t,\cdot)\|_{L_p\left(B_{R}\right)} \int_{B_{\delta_k}^c} \nu_t(dy) \leq Nc^{-1}\frac{2^{k}}{R} \|u(t,\cdot)\|_{L_p\left(B_{R}\right)}.
\end{align*}
Thus, by taking $c=\varepsilon^3/N$, we have the desired estimate.
The lemma is proved.
\end{proof}

Recall that
   \begin{equation*}
       \bT_{\kappa} (t)u(x) = \sup_{R>0} \kappa^\sigma R^{\sigma-d} \int_{B_{\kappa R}^c} \int_{B_{R}} |u(x+y+z)| dz\nu_t(dy).
   \end{equation*}

\begin{lemma}
Let $p\in(1,\infty)$ and $R>0$. Then we have
    \begin{equation} \label{eq11221518}
        \cT_{p,\kappa}^R u(t,x) \leq N(p,\Lambda) (\bT_\kappa(t)|u|^{p}(x))^{1/p}.
    \end{equation}
\end{lemma}

\begin{proof}
    By H\"older's inequality and \eqref{equiv},
    \begin{align*}
         &(\cT_{p,\kappa}^R u)(t,x) 
         \\
         &= \kappa^\sigma R^{\sigma-d/p} \int_{B_{\kappa R}^c} \|u(t,\cdot)\|_{L_p(B_R(x+y))} \nu_t(dy)
         \\
         &\leq \kappa^\sigma R^{\sigma-d/p} \left(\int_{B_{\kappa R}^c} \int_{B_R} |u(t,x+y+z)|^p dz \nu_t(dy) \right)^{1/p} \left( \int_{B_{\kappa R}^c} \nu_t(dy) \right)^{1/p'}
         \\
         &\leq N \left(\kappa^\sigma R^{-\sigma} \int_{B_{\kappa R}^c} \int_{B_R} |u(t,x+y+z)|^p dz \nu_t(dy) \right)^{1/p} \leq N (\bT_\kappa(t)|u|^{p}(x))^{1/p},
    \end{align*}
    where $p'=p/(p-1)$. The lemma is proved.
\end{proof}

   In Lemmas \ref{lem10111620}--\ref{lem_TLP}, we prove an $L_p$-boundedness of the operator $\bT_{\kappa} (t)u$. Our strategy is to first introduce a measure
   \begin{equation} \label{eq12091506}
       \mu_t(dy) := \nu_t(dy) + |y|^{-d-\sigma} dy,
   \end{equation}
   and construct from it a family of tail measures that satisfies the assumptions of \cite[Theorem A]{DR86} (see Lemma \ref{lem9251654}). We then prove that the boundedness of the associated maximal operator
   \begin{equation} \label{eq12091513}
       \bT_{\kappa}^\mu (t)u(x) := \sup_{R>0} \kappa^\sigma R^{\sigma-d} \int_{B_{\kappa R}^c} \int_{B_{R}} |u(x+y+z)| dz\mu_t(dy).
   \end{equation}
   This particular choice of $\mu_t(dy)$ is motivated by the fact that its lower bound
   \begin{equation} \label{eq12091414}
       \mu_t(B_r^c) \geq \int_{B_r^c} |y|^{-d-\sigma} dy \geq N(d,\sigma) r^{-\sigma}
   \end{equation}
   is necessary for the argument. We also notice that $\mu_t$ satisfies Assumption \ref{levy}.

\begin{lemma} \label{lem10111620}
Let $\sigma\in(0,2)$ and $r>0$. Suppose that $\nu_t$ satisfies Assumption \ref{upper}. Then, for each $t>0$,
    \begin{equation} \label{eq9251620}
         \frac{1}{\mu_t(B_r^c)} \left|\int_{|y|>r} \left(1-e^{i\xi\cdot y} \right) \mu_t(dy) \right| \leq Nr^{\sigma/2}|\xi|^{\sigma/2},
    \end{equation}
    where $\mu_t$ is defined by \eqref{eq12091506}, and $N=N(\sigma,\Lambda)$ is independent of $\xi$ and $r$.
\end{lemma}

\begin{proof}
    When $r|\xi|\geq1$, then it is obvious that
    \begin{equation*}
        \frac{1}{\mu_t(B_r^c)} \left|\int_{|y|>r} \left(1-e^{i\xi\cdot y} \right) \mu_t(dy)\right| \leq \frac{2}{\mu_t(B_r^c)} \int_{|y|>r} \mu_t(dy) = 2 \leq 2r^{\sigma/2}|\xi|^{\sigma/2}.
    \end{equation*}
    Thus, it suffices to consider the case when $r|\xi|<1$. Since $|1-e^{i\xi\cdot y}| \leq \min\{|\xi| |y|,1\}$, it follows from \eqref{eq9251355} with $\mu_t$ instead of $\nu_t$ that
    \begin{align*}
        \left|\int_{|y|>r} \left(1-e^{i\xi\cdot y} \right) \mu_t(dy)\right| &\leq |\xi| \int_{r<|y|<|\xi|^{-1}} |y| \mu_t(dy) + 2 \int_{|y|\geq |\xi|^{-1}} \mu_t(dy)
        \\
        &\leq |\xi| r^{-\sigma/2} \int_{|y|<|\xi|^{-1}} |y|^{1+\sigma/2} \mu_t(dy) + 2 \int_{|y|\geq |\xi|^{-1}} \mu_t(dy) 
        \\
        &\leq N r^{-\sigma/2}|\xi|^{\sigma/2}.
    \end{align*}
    It remains to apply \eqref{eq12091414}. The lemma is proved.
\end{proof}

\begin{lemma}
    Let $r>0$, and denote
    \begin{equation*}
        m^1_r(\xi) = \frac{1}{\omega_d r^d}\int_{\bR^d} e^{i\xi\cdot y} 1_{B_r}(y) dy
    \end{equation*}
    where $\omega_d$ is the volume of the $d$-dimensional unit ball. Then we have
    \begin{equation} \label{eq9242020}
        |m^1_r(\xi)-1| \leq N|r\xi|^2
    \end{equation}
    and
    \begin{equation} \label{eq9242021}
        |m^1_r(\xi)| \leq N|r\xi|^{-(d+1)/2},
    \end{equation}
    where $N=N(d)$.
\end{lemma}

\begin{proof}
Since this result is classical, we only give an outline of the proof.

By a change of variables, one can observe that $m^1_r(\xi) = m^1_1(r\xi)$. Thus, we only need to consider the case when $r=1$.

First, we prove \eqref{eq9242020}. Since $1_{B_r(y)}dy$ is symmetric,
\begin{equation*}
    m_1^1(\xi) -1 = \frac{1}{\omega_d} \int_{B_1} \frac{1}{2}\left( e^{i\xi\cdot y} + e^{-i\xi\cdot y} -2 \right) dy.
\end{equation*}
Thus, by using $\left| e^{i\xi\cdot y} + e^{-i\xi\cdot y} -2 \right| \leq |\xi\cdot y|^2 \leq |\xi|^2|y|^2$, we deduce
\begin{equation*}
    |m_1^1(\xi)-1| \leq \frac{|\xi|^2}{2\omega_d} \int_{B_1} y^2 dy \leq N|\xi|^2.
\end{equation*}
Hence, \eqref{eq9242020} is obtained.

Next, we consider \eqref{eq9242021}.
    The Fourier transform of the unit ball is given by
    \begin{equation} \label{eq9242036}
        m^1_1(\xi) = \frac{(2\pi)^{d/2}}{\omega_d} \frac{J_{d/2}(|\xi|)}{|\xi|^{d/2}} = 2^{d/2}\Gamma(d/2+1) \frac{J_{d/2}(|\xi|)}{|\xi|^{d/2}},
    \end{equation}
    where $J_{d/2}$ is the Bessel function of order $d/2$, which is defined as
    \begin{equation*}
        J_{d/2}(z)=\sum_{j=0}^\infty \frac{(-1)^j}{j!}\frac{1}{\Gamma(j+d/2+1)}\left(\frac{z}{2}\right)^{2j+d/2}.
    \end{equation*}
Then \eqref{eq9242021} easily follows from a well-known asymptotic behavior
    \begin{equation} \label{eq9242037}
        |J_{d/2}(|\xi|)| \leq |\xi|^{-1/2}.
    \end{equation}
As a final remark, we refer the reader to \cite[Appendices B.5 and B.7]{G14} for \eqref{eq9242036} and \eqref{eq9242037}. The lemma is proved.
\end{proof}

The following is taken from \cite[Theorem A]{DR86}.

\begin{lemma} \label{lem9251654}
    Let $p\in(1,\infty]$, $a,C>0$, and $\{\mu_k\}_{k=-\infty}^\infty$ be a sequence of nonnegative Borel measures in $\bR^d$ such that $\|\mu_k\|=1$. Assume that $a_k>0$ and
    \begin{equation*}
        c:=\inf_{k\in\bZ}\frac{a_{k+1}}{a_k}>1.
    \end{equation*}
    Suppose that
    \begin{align*}
        |\widehat{\mu}_k(\xi) - 1| &\leq C|a_{k+1}\xi|^a,
        \\
        |\widehat{\mu}_k(\xi)| &\leq C|a_{k}\xi|^{-a}.
    \end{align*}
    Then, the maximal operator $Tu:=\sup_{k\in\bZ}|u*\mu_k|$ is bounded in $L_p(\bR^d)$. Moreover,
    \begin{equation*}
        \|Tu\|_{L_p(\bR^d)} \leq N \|u\|_{L_p(\bR^d)},
    \end{equation*}
    where $N=N(a,c,d,p,C)$.
\end{lemma}

\begin{lemma} \label{lem_TLP}
    Let $p\in(1,\infty)$, $u\in L_p(\bR^d)$, $t\in(0,T)$, and $\kappa\in(0,1)$. Suppose that $\nu_t$ satisfies Assumption \ref{upper}. Then, we have
  \begin{equation*}
      \|\bT_{\kappa}(t)u\|_{L_p(\bR^d)} \leq N\|u\|_{L_p(\bR^d)},
  \end{equation*}
  where $N=N(d,p,\sigma,\Lambda)$ is independent of $\kappa$.
\end{lemma}

\begin{proof}
Since $t>0$ is fixed, we omit $t$; for instance, $\bT_{\kappa} u=\bT_{\kappa}(t)u$.
First, we show that $\bT_{\kappa} u$ is measurable. Note that for each $R>0$, by H\"older's inequality and \eqref{equiv},
\begin{align*}
   &\kappa^\sigma R^{\sigma-d} \int_{B_{\kappa R}^c} \int_{B_R} |u(x+y+z)| dz \nu_t(dy) 
   \\
   &\leq \kappa^\sigma R^{\sigma-d/p} \int_{B_{\kappa R}^c} \left(\int_{B_R} |u(x+y+z)|^p dz\right)^{1/p} \nu_t(dy)
   \\
   &\leq N \kappa^\sigma R^{-d/p} \|u\|_{L_{p}(\bR^d)} <\infty.
\end{align*}
Thus, since $R\to \nu_t(B_{\kappa R}^c)$ is left continuous, by the dominated convergence theorem,
\begin{equation*}
    R\to \kappa^\sigma R^{\sigma-d} \int_{B_{\kappa R}^c} \int_{B_R} |u(x+y+z)| dz \nu_t(dy)
\end{equation*}
is left continuous, which yields the desired claim.

Since $\nu_t(dy) \leq \mu_t(dy)$ where $\mu_t(dy)$ is defined by \eqref{eq12091506}, it suffices to show that
  \begin{equation*}
      \|\bT_{\kappa}^\mu(t)u\|_{L_p(\bR^d)} \leq N\|u\|_{L_p(\bR^d)}.
  \end{equation*}
Let $R>0$ be given, and take $j\in\bZ$ so that $2^{j}<R \leq 2^{j+1}$. Then,
\begin{equation} \label{eq9251548}
    R^{-d}\int_{B_{R}} |u(x+y+z)| dz \leq N \aint_{B_{2^{j+1}}} |u(x+y+z)| dz,
\end{equation}
where $N$ is independent of $R$. Note that by \eqref{upper_mu}, one can easily show that
\begin{equation*}
    \mu_t(B_r^c) \leq N(d,\sigma,\Lambda) r^{-\sigma}.
\end{equation*}
Using this, for any function $h$,
\begin{align*}
    \kappa^{\sigma}R^{\sigma} \int_{B_{\kappa R}^c} h(y) \mu_t(dy) \leq N\kappa^{\sigma}2^{j\sigma} \int_{B_{\kappa 2^j}^c} h(y) \mu_t(dy) \leq N\aint_{B_{\kappa 2^{j}}^c} h(y) \mu_t(dy).
\end{align*}
This together with \eqref{eq9251548} leads to
\begin{align} \label{eq9252331}
    &\kappa^{\sigma}R^{\sigma-d} \int_{B_{\kappa R}^c} \int_{B_{R}} |u(x+y+z)| dz\mu_t(dy) \nonumber
    \\
    &\leq  N \aint_{B_{\kappa 2^{j}}^c} \aint_{B_{2^{j+1}}} |u(x+y+z)| dz \mu_t(dy) =: N\cT_\kappa^j u(x).
\end{align}

Let
\begin{equation*}
    \mu^1_j(dz) := \frac{1}{\omega_d2^{(j+1)d}}1_{B_{2^{j+1}}}(z) dz, \quad  \mu^2_j(dy) := \frac{1}{\mu_t(B_{\kappa 2^{j}}^c)}1_{B_{\kappa 2^{j}}^c}(y)\mu_t(dy)
\end{equation*}
be two Borel measures in $\bR^d$ such that $\|\mu^1_j\|=\|\mu^2_j\|=1$. Then $\cT_\kappa^j u$ can be represented as
\begin{equation*}
    \cT_\kappa^j u(x)= u*\mu_{j}(x),
\end{equation*}
where $\mu_{j}=\mu^1_j*\mu^2_{j}$ denotes the convolution of the two measures. Let $a:=\min\{\sigma/2,(d+1)/2\}$. Then by \eqref{eq9251620} and \eqref{eq9242020}, when $2^j|\xi|<1$,
\begin{align*}
    |\widehat{\mu}_{j}(\xi)-1| &\leq |\widehat{\mu}_{j}^2(\xi)||\widehat{\mu}_j^1(\xi)-1| + |\widehat{\mu}_{j}^2(\xi)-1| 
    \\
    &\leq N2^{2j}|\xi|^2 + N\kappa^{\sigma/2}2^{j\sigma/2}|\xi|^{\sigma/2} \leq N2^{j\sigma/2}|\xi|^{\sigma/2} \leq N|2^{j+1}\xi|^{a}.
\end{align*}
Note that we used the condition $\kappa<1$.
When $2^j|\xi|\geq1$, one can just use $\|\mu_{k,j}\|=\|\mu^1_j\| \times \|\mu^2_{j}\|=1$ to obtain that
\begin{equation*} 
    |\widehat{\mu}_{j}(\xi)-1| \leq 2 \leq N|2^{j+1}\xi|^{a}.
\end{equation*}
Thus, for any $\xi\in\bR^d$, we have
\begin{equation} \label{eq9251649}
    |\widehat{\mu}_{j}(\xi)-1| \leq N|2^{j+1}\xi|^a.
\end{equation}
By \eqref{eq9242021}, one can easily show that
\begin{equation} \label{eq9251650}
     |\widehat{\mu}_{j}(\xi)|= |\widehat{\mu}_j^1(\xi) \widehat{\mu}_{j}^2(\xi)| \leq |\widehat{\mu}_j^1(\xi)| \leq \min\{1,N|2^{j}\xi|^{-(d+1)/2}\} \leq N|2^{j}\xi|^{-a}.
\end{equation}
By \eqref{eq9251649} and \eqref{eq9251650}, we can apply Lemma \ref{lem9251654} with $a_j=2^j$, which leads to
\begin{equation*}
    \left\|\sup_{j\in\bZ} \cT_\kappa^ju \right\|_{L_p(\bR^d)} \leq N \|u\|_{L_p(\bR^d)},
\end{equation*}
where $N$ is independent of $\kappa$. This, \eqref{eq12091513}, and \eqref{eq9252331} yield the desired result.
The lemma is proved.
\end{proof}

\begin{remark}
    It follows from Lemma \ref{lem_TLP} that for the boundedness of the operator $\bT_\kappa(t)$ does not require the nondegenerate condition \eqref{nonde}.
\end{remark}

The following is a Sobolev embedding theorem, which is taken from \cite[Lemma A.6]{DL23}.

\begin{lemma} \label{lem11072224}
    Let $\sigma\in(0,2)$, $T\in(0,\infty], p\in (1,\infty),$ and $u\in \bH_{p,0}^\sigma(T)$.

        (i) Suppose that $p<d/\sigma+1$, and take $q\in(p,\infty)$ such that
    \begin{equation} \label{eq11072207}
        1/q=1/p-\sigma/(d+\sigma).
    \end{equation}
    Then for any $l\in [p,q]$,
    \begin{equation*}
        \|u\|_{L_l(\bR^d_T)} \leq N\|u\|_{\bH_{p}^\sigma(T)},
    \end{equation*}
    where $N=N(d,\sigma,p,l,T)$.

    (ii) If $p=d/\sigma+1$, then the same estimate holds with $l\in[p,\infty)$. 

    (iii) If $p>d/\sigma+1$, then for $\tau:=\sigma-(d+\sigma)/p$,
        \begin{equation*}
        \|u\|_{C^{\tau/\sigma,\tau}(\bR^d_T)} \leq N\|u\|_{\bH_{p}^\sigma(T)},
    \end{equation*}
    where $N=N(d,\sigma,p,T)$.
\end{lemma}

\begin{corollary}
    Let $t_0>0$, $R>0$, $\sigma\in(0,2)$, $T\in(0,\infty], p\in (1,\infty)$, $\vartheta:=\sigma-(d+\sigma)/p$, and $u\in \bH_{p,0}^\sigma(t_0-R^\sigma,t_0)$. Take $\zeta_0\in C_c^\infty(B_R)$ such that $\zeta_0=1$ in $B_{R/2}$.
  
  (i) Let
    \begin{equation*}
        \begin{cases}
            l\in [p,q] &\text{ if } p<d/\sigma+1
            \\
            l\in[p,\infty) &\text{ if } p=d/\sigma+1
            \\
            l\in[p,\infty] &\text{ if } p>d/\sigma+1,
        \end{cases}
    \end{equation*}
    where $q$ is defined by \eqref{eq11072207}. Then we have
    \begin{align} \label{eq11072241}
    &(|u|^{l})^{1/l}_{Q_{R/2}(t_0,0)} \nonumber
    \\
    &\leq NR^{\vartheta} \left(\|\partial_t(\zeta_0 u)\|_{L_p((t_0-R^\sigma,t_0)\times \bR^d)} + \|(-\Delta)^{\sigma/2}(\zeta_0 u)\|_{L_p((t_0-R^\sigma,t_0)\times \bR^d)} \right),
\end{align}
where $N=N(d,p,\sigma)$ is independent of $T$.

(ii) If $p>d/\sigma+1$, then
\begin{align} \label{eq11290052}
    &[u]_{C^{\vartheta/\sigma,\vartheta}(Q_{R/2}(t_0,0))} \nonumber
    \\
    &\leq N\left(\|\partial_t(\zeta_0 u)\|_{L_p((t_0-R^\sigma,t_0)\times \bR^d)} + \|(-\Delta)^{\sigma/2}(\zeta_0 u)\|_{L_p((t_0-R^\sigma,t_0)\times \bR^d)} \right),
\end{align}
where $N=N(d,p,\sigma)$.
\end{corollary}

\begin{proof}
    By the change of variables $(t,x)\to (R^\sigma t,Rx)$, it suffices to deal with the case when $R=1$. 
    
    $(i)$ 
    First, when $p\leq d/\sigma+1$, by Lemma \ref{lem11072224} $(i)$ or $(ii)$, 
    \begin{align*}
    (|u|^{l})^{1/l}_{Q_{1/2}(t_0,0)} &\leq \|\zeta_0 u\|_{L_l((t_0-1,t_0)\times \bR^d)} \nonumber
    \\
    &\leq N \|\zeta_0 u\|_{L_p((t_0-1,t_0)\times \bR^d)} + N\|(-\Delta)^{\sigma/2}(\zeta_0 u)\|_{L_p((t_0-1,t_0)\times \bR^d)} \nonumber
    \\
    &\quad + N\|\partial_t(\zeta_0 u)\|_{L_p((t_0-1,t_0)\times \bR^d)} .
\end{align*}
It remains to apply \eqref{eq110722228} together with a shift of the coordinates in the time direction.

Next, we consider the case when $p>d/\sigma+1$. Since $(\zeta_0u)(t_0-1,x)=0$, for $(t,x)\in Q_{1/2}(t_0,0)$,
\begin{align*}
      |(\zeta_0u)(t,x)| &= |t- (t_0-1)|^{\tau/\sigma}\frac{|(\zeta_0u)(t,x) - (\zeta_0u)(t_0-1,x)|}{|t- (t_0-1)|^{\tau/\sigma}} 
      \\
      &\leq N[\zeta_0u]_{C^{\tau/\sigma,\sigma}((t_0-1,t_0)\times\bR^d)},
  \end{align*}
  which yields that
  \begin{align*}
      (|u|^{l})^{1/l}_{Q_{1/2}(t_0,0)} &\leq \|\zeta_0 u\|_{L_\infty(Q_{1/2}(t_0,0))} \leq N\|\zeta_0u\|_{C^{\tau/\sigma,\sigma}((t_0-1,t_0)\times\bR^d)}.
  \end{align*}
  Then it remains to repeat the above argument with Lemma \ref{lem11072224} $(iii)$.

  $(ii)$
  For \eqref{eq11290052}, one just needs to use Lemma \ref{lem11072224} $(iii)$ together with
 \begin{align*}
    [u]_{C^{\vartheta/\sigma,\sigma}(Q_{1/2}(t_0,0)} \leq [\zeta_0 u]_{C^{\vartheta/\sigma,\sigma}((t_0-1,t_0)\times \bR^d)}.
\end{align*}
The corollary is proved.
\end{proof}

In the following lemma, $\cC_R(t,x)$ and $\widehat{\cC_R}(t,x)$ are introduced in \eqref{eq10152356}.

\begin{lemma} \label{lem10161553}
    Let $\gamma\in(0,1)$, and $E\subset F\subset \bR^d_T$ with $|E|<\infty$. Suppose that if 
    \begin{equation*}
        |\cC_R(t,x) \cap E| \geq \gamma |\cC_R(t,x)|,
    \end{equation*}
    for $R>0$ and $(t,x)\in \bR^d_T$, then
    \begin{equation*}
        \widehat{\cC_R}(t,x) \subset F.
    \end{equation*}
    Then we have
    \begin{equation*}
        |E|\leq N(d)\gamma|F|.
    \end{equation*}
\end{lemma}

\begin{proof}
    See \cite[Lemma A.20]{DK19}.
\end{proof}





\end{document}